\documentclass[a4paper,aos,reqno,11pt]{imsart}
\usepackage[utf8]{inputenc}
\usepackage{amsmath}
\usepackage{amsfonts,amssymb}
\usepackage{dsfont}
\usepackage{graphicx}
\usepackage[dvipsnames]{xcolor}
\usepackage{caption,mathrsfs,booktabs}
\usepackage{nicefrac}
\usepackage{enumitem}
\usepackage{algorithm,algorithmic}

\RequirePackage[colorlinks,citecolor=blue,urlcolor=blue,breaklinks=true]{hyperref}
\usepackage{cleveref}
\usepackage{autonum}

\usepackage[round,sort]{natbib}
\newcommand\prob{\mathbb{P}}
\newcommand\indic{\mathds{1}}

\newcommand\rSigma{\textbf{\textup{r}}_{\scriptscriptstyle\bSigma}}
\newcommand\rtSigma{\textbf{\textup{r}}_{\scriptscriptstyle\widetilde\bSigma}}

\makeatletter
\newcommand*{\Xbar}{}%
\DeclareRobustCommand*{\Xbar}{%
  \mathpalette\@Xbar{}%
}
\newcommand*{\@Xbar}[2]{%
  \sbox0{$#1\mathbf{X}\m@th$}%
  \sbox2{$#1\bs X\m@th$}%
  \rlap{%
    \hbox to\wd2{%
      \hfill
      $\overline{%
        \vrule width 0pt height\ht0 %
        \kern\wd0 %
      }$%
    }%
  }%
  \copy2 %
}
\makeatother

\makeatletter
\newcommand*{\Zbar}{}%
\DeclareRobustCommand*{\Zbar}{%
  \mathpalette\@Zbar{}%
}
\newcommand*{\@Zbar}[2]{%
  \sbox0{$#1\mathbf{Z}\m@th$}%
  \sbox2{$#1\bs Z\m@th$}%
  \rlap{%
    \hbox to\wd2{%
      \hfill
      $\overline{%
        \vrule width 0pt height\ht0 %
        \kern\wd0 %
      }$%
    }%
  }%
  \copy2 %
}
\makeatother

\newcommand\calS{\mathcal{S}}
\newcommand\calO{\mathcal{O}}
\newcommand\calI{\mathcal{I}}

\newcommand\bs{\boldsymbol}
\newcommand\bmu{\boldsymbol\mu}
\newcommand\bxi{\boldsymbol\xi}
\newcommand\bzeta{\boldsymbol\zeta}
\newcommand\bSigma{\boldsymbol\Sigma}
\newcommand\btSigma{\widetilde{\boldsymbol\Sigma}}

\newcommand\bfA{\mathbf A}
\newcommand\bfB{\mathbf B}

\newcommand\bfP{\mathbf P}

\newcommand\bfU{\mathbf U}
\newcommand\bfV{\mathbf V}

\newcommand\bfM{\mathbf M}
\newcommand\be{\boldsymbol e}

\newcommand\bZ{\boldsymbol Z}

\newcommand\bV{\boldsymbol V}
\newcommand\bX{\boldsymbol X}

\newcommand\bY{\boldsymbol Y}

\newcommand\bv{\boldsymbol v}

\newcommand\RR{\mathbb R}

\newcommand\bfI{\mathbf I}

\newcommand\barx{\Xbar}
\newcommand\barz{\Zbar}
\newcommand\barxi{\overline\bxi}

\def\tr{\mathop{\textup{Tr}}}

\def\hat{\widehat}

\def\bmuGM{\hat\bmu^{\sf GM}}
\def\bmuGMV{\hat\bmu^{{\sf GM}}_V}
\def\bmuSDR{\hat\bmu{}^{{\sf SDR}}}
\def\Proj{{\tt P}}
\def\ProjV{{\tt P}_V}

\def\Err{\textup{Err}}
\def\sfC{\mathsf C}

\usepackage{amsthm}
\newtheorem{theorem}{Theorem}
\newtheorem{definition}{Definition}
\newtheorem{proposition}{Proposition}
\newtheorem{lemma}{Lemma}

\newcommand\med{\hat\bmu^{{\sf GM}}}

\begin{document}

\begin{frontmatter}

\title{Nearly minimax robust estimator of the mean vector by
iterative spectral dimension reduction}
\runtitle{Robust mean estimation by SDR}

\begin{aug}

\author
    {
    \fnms{Amir-Hossein}
    \snm{Bateni}\ead[label=e1]{Amirhossein.bateni@ensae.fr}},
\author
    {
    \fnms{Arshak}
    \snm{Minasyan}\ead[label=e2]{arshak.minasyan@ensae.fr}
    }
\and
\author
    {
    \fnms{Arnak S.}
    \snm{Dalalyan}
    \ead[label=e3]{arnak.dalalyan@ensae.fr}
    }
\address{
    CEREMADE, Université Paris Dauphine - PSL\\
    Place du Maréchal de Lattre de Tassigny,
    75775 Paris Cedex 16\\
    \printead{e1}
    }
\address{
    CREST, ENSAE, IP Paris\\
    5 avenue Henry Le Chatelier,
    91764 Palaiseau\\
    \printead{e2,e3}
    }

\runauthor{A. Bateni, A. Minasyan, A. Dalalyan}

\end{aug}

\begin{abstract}
We study the problem of robust estimation of the mean
vector of a sub-Gaussian distribution. We introduce an
estimator based on spectral dimension reduction (SDR)
and establish a finite sample upper bound on its error
that is minimax-optimal up to a logarithmic
factor. Furthermore, we prove that the breakdown point
of the SDR estimator is equal to $1/2$, the highest
possible value of the breakdown point. In addition, the
SDR estimator is equivariant by similarity transforms
and has low computational complexity. More precisely,
in the case of $n$ vectors of dimension $p$---at most
$\varepsilon n$ out of which are adversarially
corrupted---the SDR estimator has a squared error of order
$(\nicefrac{\rSigma}{n} + \varepsilon^2\log(1/\varepsilon)
){\log p}$ and a running time of order $p^3 + n p^2$.
Here, $\rSigma\le p$ is the effective rank of the
covariance matrix of the reference distribution.
Another advantage of the SDR estimator is that it does
not require knowledge of the contamination rate
and does not involve sample splitting. We also investigate
extensions of the proposed algorithm and of the obtained
results in the case of (partially) unknown covariance matrix.
\end{abstract}



\end{frontmatter}

\section{Introduction}

Robust estimation of a finite-dimensional parameter is a classical
problem in statistics. The broad goal of robust estimation is to
design statistical procedures that are not very sensitive
to small changes in data or to small departures from the
modeling assumptions. A typical example, extensively
studied in the literature, and considered in the present
work, is when the data set contains outliers.

The literature on robustness to outliers in parametric
estimation is very rich; it would be impossible to review
here all the important contributions. For an in-depth
exposition of by now classical results and approaches, such
as the influence function, the breakdown point and the
efficiency, we refer to the books \citep{rousseeuw2011robust,
huber2011robust,maronna2006robust}. In their vast majority,
these well-established approaches treated the dimension of
the parameter as a fixed and small constant. This simple
setting was convenient for mathematical
analysis and for computational purposes,
but somewhat disconnected from many practical situations.
Furthermore, it was hiding some fascinating phenomena that
emerge only when the dimension is considered as
a parameter that might be large, in the same way as the
sample size.

\begin{figure}
    \centering
    \includegraphics[width=0.3\textwidth]{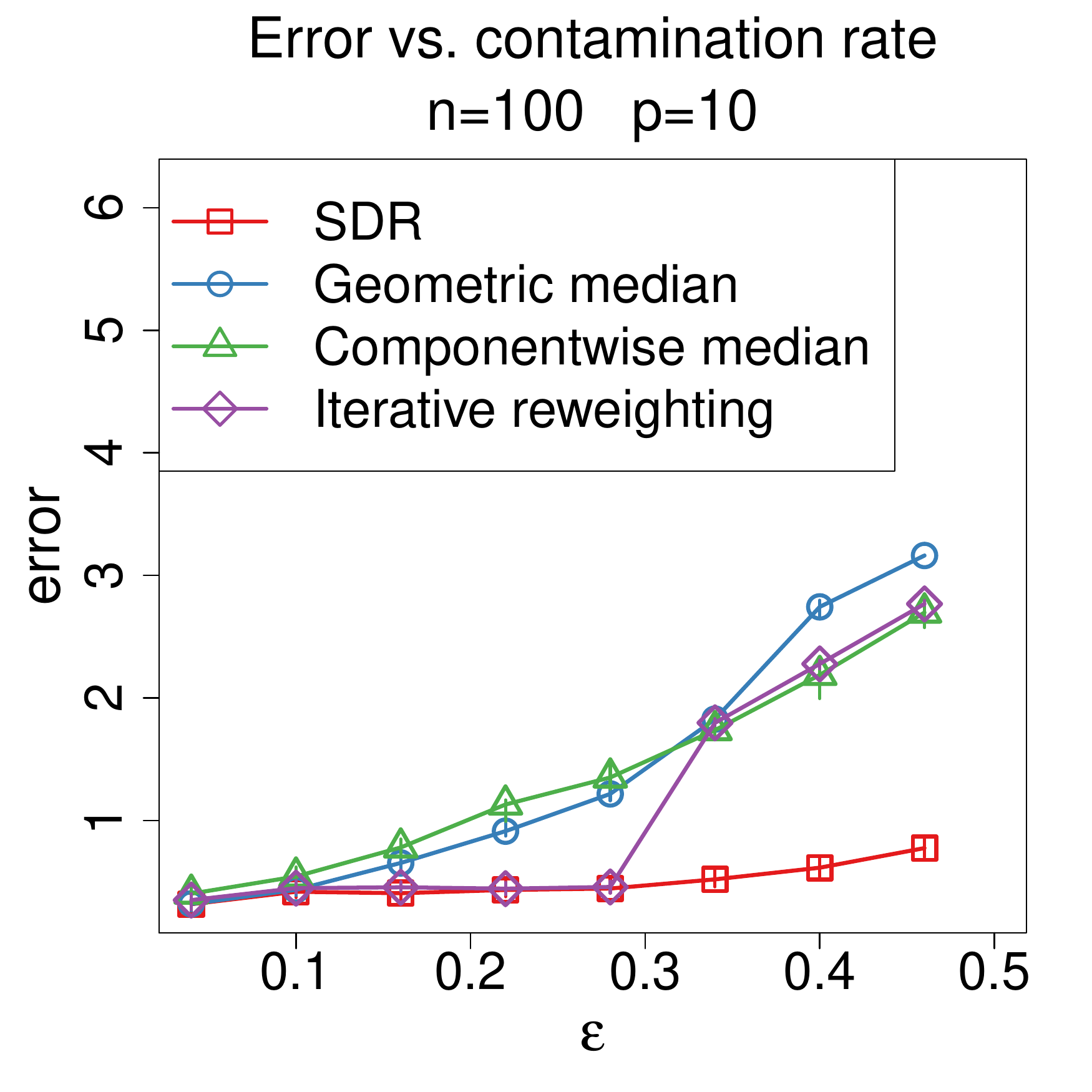}
    \includegraphics[width=0.3\textwidth]{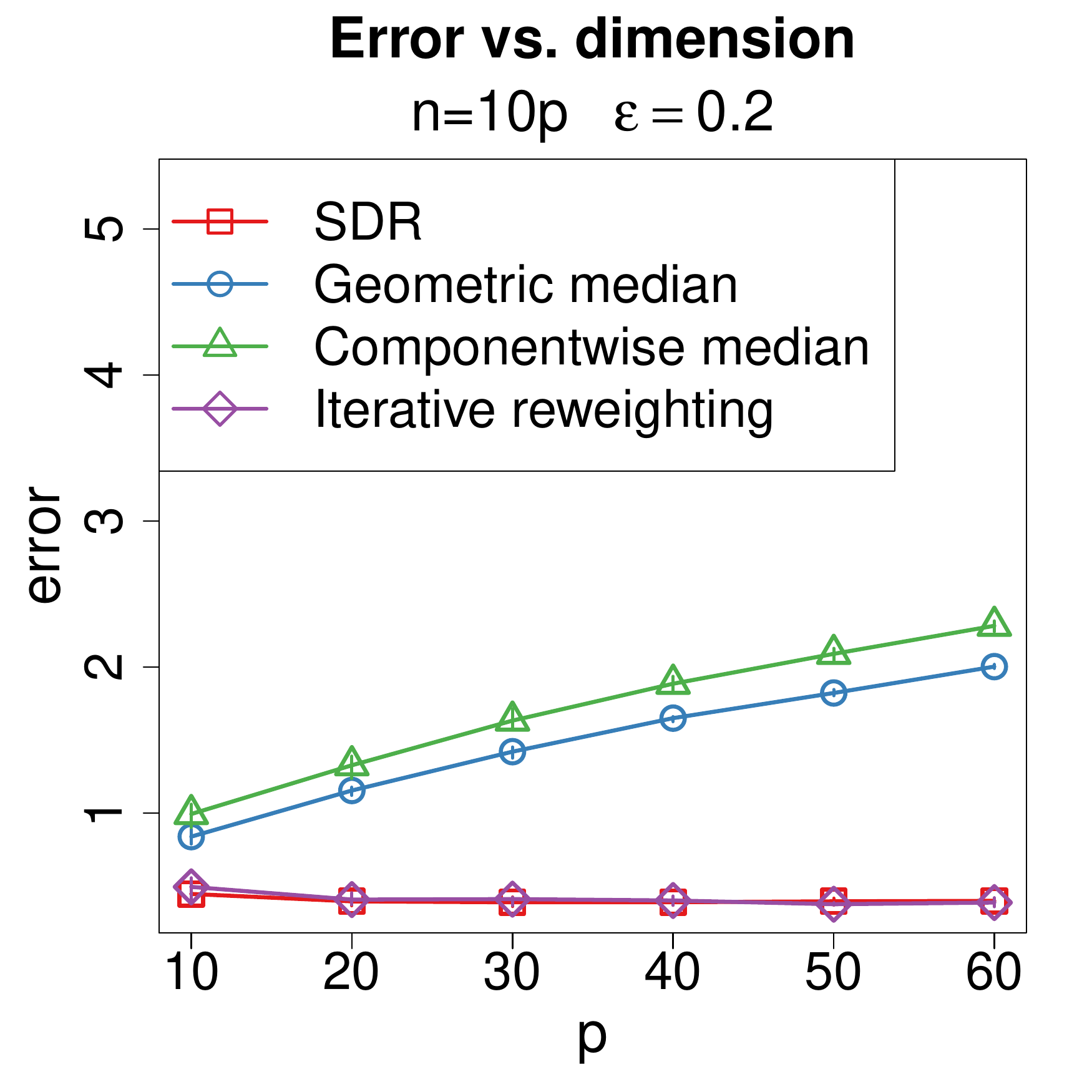}
    \includegraphics[width=0.3\textwidth]{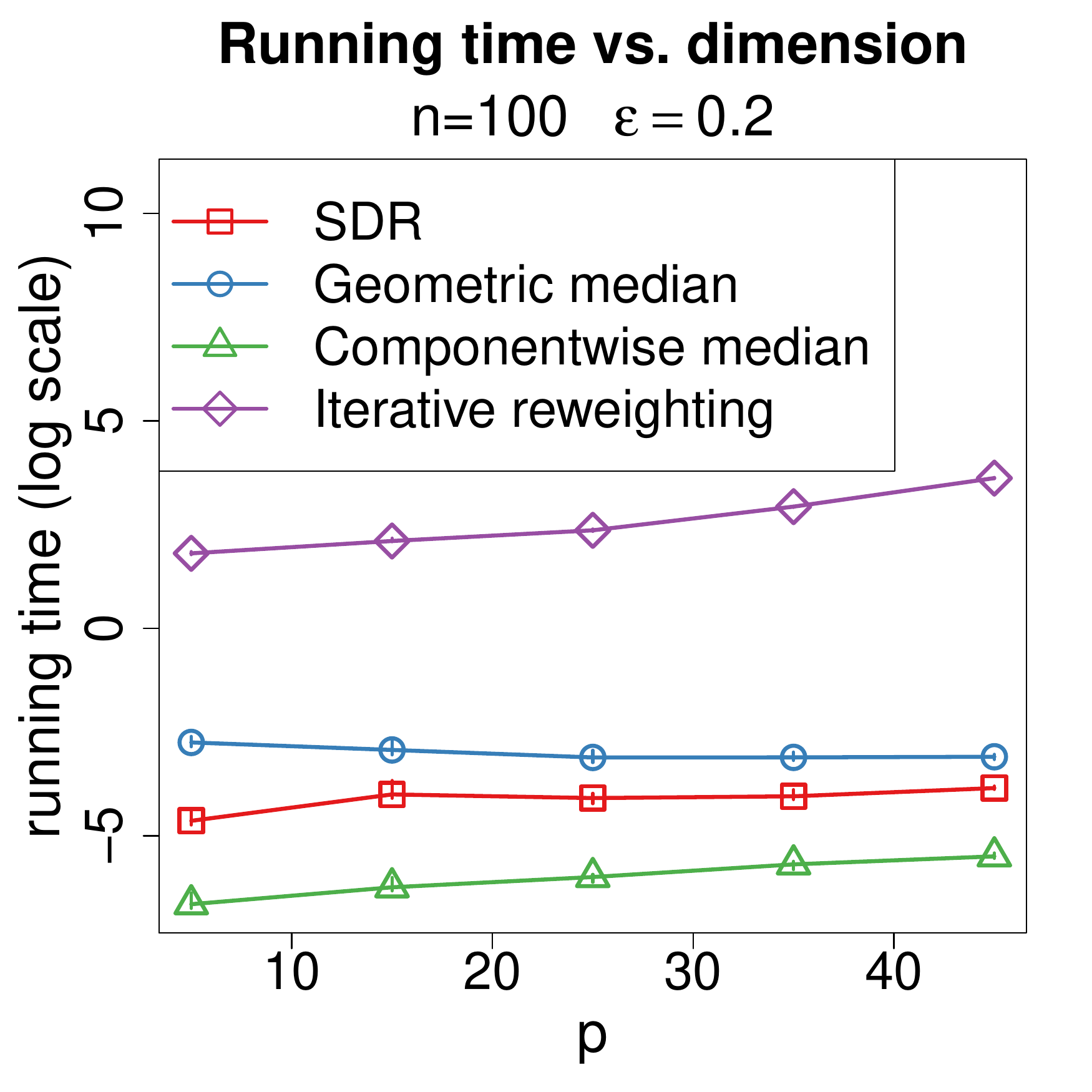}
    \caption{\small Plots that help to visually compare four
    robust estimators: SDR (our estimator), geometric
    median (GM) given by \eqref{GM}, componentwise
    median (CM), iteratively reweighted mean (IRM) of
    \citep{dalalyan2020allinone}. The
    first two plots show that SDR is as accurate as IRM
    for small $\epsilon$, with SDR outperforming IRM
    for $\varepsilon$ close to $1/2$. IRM and SDR are
    naturally much more accurate than GM and CM. The last
    plot shows that the running time of SDR is comparable
    to that of GM and is much smaller than that of IRM.
    More details on these experiments are provided in
    \Cref{sec:5}. }
    \label{fig:intro}
\end{figure}

More recently, \citep{chen2018} considered the problem of
estimating the mean and the covariance matrix of a Gaussian
distribution in the high-dimensional setting. The authors
namely uncovered a new phenomenon: under the Huber
contamination, the componentwise median is not minimax-rate
optimal whereas the Tukey median is. More precisely, if a
$p$-dimensional mean vector is to be estimated from $n$
independent vectors drawn from the mixture distribution
$(1-\varepsilon)\mathcal N_p(\bmu,\bSigma) +\varepsilon
\mathbf Q$ (where $\varepsilon\in(0,1/2)$ is the rate of
contamination and $\mathbf Q$ is the unknown distribution
of outliers), then the mean squared error of the
componentwise median is of order $p/n + p\varepsilon^2$
while that of Tukey's median is of order $p/n +
\varepsilon^2$. This extra factor $p$ in front of
$\varepsilon^2$ has been proven in \citep{LaiRV16}
to be present in the error of another widely used
robust estimator of the mean, the  geometric median
\citep{Minsker}. Thus, as long as only statistical
properties of the estimators are considered, Tukey's
median is thus superior to its competitors, the
componentwise and the geometric medians.  However,
the componentwise and the geometric medians are better
than the Tukey's  median in terms of the breakdown
point: their breakdown point is equal to $1/2$
\citep{Lopuhaa} whereas that of Tukey's median is
$1/3$ \citep{DonGasko}. This is one of the appealing
phenomena taking place in the high dimensional setting.

Another specificity of the high dimensional setting
uncovered by \citep{chen2018} was the lack of computational
tractability of the estimators that are statistically
optimal. Indeed, Tukey's median is computationally
intractable, but minimax-rate optimal, whereas the
componentwise and the geometric medians are
computationally tractable but statistically suboptimal.
This observation led to the development of a number of
computationally tractable estimators having an error
with a better dependence on dimension than that of
Tukey's median \citep{DiakonikolasKKL16,LaiRV16,DiakonikolasKK017, DiakonikolasKK018,DongH019,dalalyan2020allinone}.
In particular, most estimators introduced in these papers
allow to conciliate computational tractability (\textit{i.e.}, are computable in time polynomial in $n,p,1/\varepsilon$) and
statistical optimality up to logarithmic factors.

The goal of this paper is  to make a step forward by
designing an estimator which is not only nearly rate
optimal and computationally tractable, but also has a
breakdown point equal to $1/2$, which is the highest
possible value of the breakdown point. To construct the
estimator, termed iterative spectral dimension reduction
or SDR, we combine and suitably adapt ideas from
\citep{LaiRV16} and \citep{DiakonikolasKK017}. The
main underlying observation is that if we remove some
clear outliers and restrict our attention to the subspace
spanned by the eigenvectors of the sample covariance matrix
corresponding to small eigenvalues, then the sample mean
of the projected data points is a rate-optimal estimator.
This allows us to iteratively reduce the dimension and
eventually to estimate the remaining low-dimensional
component of the mean by a standard robust estimator such
as the componentwise median or the trimmed mean, see
\Cref{alg:1}.

The main contributions of this paper are methodological and
theoretical. The SDR estimator, thoroughly defined in
\Cref{sec:2}, is a fast and accurate method for robustly
estimating the mean of a set of points. It depends on one
tuning parameter, the threshold used for identifying and
removing clear outliers, and on the dimension reduction
regime. Our theoretical considerations provide some
recommendations for their choices and our numerical experiments
reported in \Cref{sec:5} confirm the relevance of these choices.
Importantly, the SDR estimator does not require as input the
rate of contamination $\varepsilon$ but only an upper bound
on $\varepsilon$. As for theoretical contributions of this
paper, we state in \Cref{sec:3}
an upper bound on the error of the SDR estimator,
showing that it is nearly minimax-rate optimal and has a
breakdown point equal to $1/2$. This is done in the general
case of a sub-Gaussian distribution with heterogeneous
covariance matrix contaminated by adversarial noise. In
\Cref{sec:4}, we further investigate the error of the SDR
estimator in the case where only an approximation to the
covariance matrix is available.

\begin{table}[]
    \centering
    \begin{tabular}{r|ccccc}
    \toprule
         & Comput. & Breakdown & knowledge
         & Squared  & knowledge \\
         & tractable & point & of $\varepsilon$
         & error rate & of $\bSigma$ or $\bSigma \propto
         \mathbf I$\\
         \toprule
         &\multicolumn{5}{c}{Gaussian distribution}\\
         \midrule
    Comp./Geom.\ Median & \color{blue}yes  & 0.5 & \color{blue}
    no
         & ${\rSigma}/{n} + \varepsilon^2 p$ & \color{blue}no  \\
    Tukey's Median & \color{red} no  & 0.33 & \color{blue} no
         &  ${\rSigma}/{n} + \varepsilon^2$ & \color{blue} no  \\
    Agnostic Mean
        & \color{blue} yes & $-$ & \color{red}yes & $(p/n)\log^3p
        + \varepsilon^2 \log p$ & \color{red} yes \\
        \midrule
        &\multicolumn{5}{c}{sub-Gaussian distribution}\\
        \midrule
    Iter.\ Reweighted Mean & \color{blue}yes & 0.28
    & \color{red} yes &  $({\rSigma}/{n}) + \varepsilon^2
    \log(1/\varepsilon)$ & \color{red} yes \\
    Iterative Filtering
        & \color{blue}yes & $-$ & \color{red} yes & $(p/n)\log^ap
        + \varepsilon^2\log(1/\varepsilon)$ & \color{red} yes \\
    SDR (this paper)
        & \color{blue}yes & 0.5 & \color{blue} no & $(\rSigma/n + \varepsilon^2\log(1/\varepsilon))\log p$ & \color{red} yes \\
    \bottomrule
    \end{tabular}
    \caption{Properties of various robust estimators. Agnostic
    mean, iteratively reweighted mean and iterative filtering
    are the estimators studied in \citep{LaiRV16},
    \citep{dalalyan2020allinone} and \citep{DiakonikolasKK017},
    respectively. The error rates reported for Tukey's median,
    componentwise median, geometric median and the agnostic
    mean have been proved for non-adversarial contamination.
    The squared error rate is provided in the case of a
    covariance matrix satisfying $\|\bSigma\|_{\rm op}=1$.}
    \label{tab:intro}
\end{table}

The papers that are the closest to the  present one are
\citep{LaiRV16}, \citep{DiakonikolasKK017} and
\citep{dalalyan2020allinone}. The spectral dimension
reduction scheme was proposed by \citep{LaiRV16} along with
an initial sample splitting step ensuring the independence of
the estimators over different subspaces.
In the case of spherical Gaussian distribution contaminated
by non-adversarial outliers, the paper states that the
proposed estimator has a squared error at most of order
$p\log^2 p\log(p/\varepsilon)/n + \varepsilon^2\log p $.
Compared to this, our results are valid in the more
general setting of sub-Gaussian distribution, with arbitrary
covariance matrix and adversarial contamination. In addition,
our estimator does not rely on sample splitting and, therefore,
has a risk with a better dependence on $p$. As compared to
the filtering method of  \citep{DiakonikolasKK017},
our estimator has the advantage of being independent of
$\varepsilon$ and our error bound is valid for every covariance
matrix and every confidence level. On the down side, our
error bound has an extra factor $\log p$ in front of
$\varepsilon^2$. We believe that this factor is an artifact
of the proof, but we were unable to remove it. Finally,
compared to the iteratively reweighted mean
\citep{dalalyan2020allinone}, the SDR estimator studied
in the present paper has a higher breakdown point, does not
require the knowledge of $\varepsilon$ and is much faster to
compute. The advantages and shortcomings of these estimators
are summarized in \Cref{tab:intro} and \Cref{fig:intro}.

\textit{Notation.}\ For any pair of integers $k$ and $d$
such that $1\le k\le d$, we denote by $\mathscr V_k^d$ the
set of all $k$-dimensional linear subspaces $V$ of
$\mathbb{R}^d$. For $V\in \mathscr V_k^d$, we write
$k=\dim(V)$ and denote by $\ProjV$ the orthogonal projection
matrix onto $V$. $\mathbb{S}^{d-1}$ stands for the unit
sphere in $\mathbb{R}^d$. For a $d\times d$ symmetric matrix
$\bfM$, we denote by $\lambda_1(\bfM),\ldots,\lambda_d(\bfM)$
its eigenvalues sorted in increasing order, and use the
notation $\lambda_{\min}(\bfM) = \lambda_1(\bfM)$,
$\lambda_{\max}(\bfM) = \lambda_d(\bfM)$, $\|\bfM\|_{\rm op}
= \max(|\lambda_{\min}(\bfM)|, |\lambda_{\max}(\bfM)|)$ and
$\tr(\bfM) = (\lambda_1 + \ldots + \lambda_d)(\bfM)$. For
any integer $n>0$, we set $[n] = \{1,\ldots,n\}$. We will
denote by $\calO\subset [n]$ the subscripts of the outliers
and by $\calI=[n]\setminus \calO$ the subscripts of inliers.
We also use notation $\log_+(x) = \max\{0, \log(x)\}$.


\begin{center}
    \begin{algorithm}[h]
    \caption{\textsf{SDR}($\bX_1,\dots,\bX_n;\bSigma,t$)}
    \label{alg:1}
    \begin{algorithmic}[1]
        \STATE\textbf{let} $p$ the dimension of $\bX_1$
        \STATE\textbf{let} $\med$ be the geometric median
        of $\bX_1,\dots,\bX_n$
        \STATE\textbf{let} $\calS \gets  \{i:\|\bX_i-\med\|
        \leq t\sqrt{p}\}$
        \STATE\textbf{let} $\barx_{\calS}$ be the sample
        mean of the filtered sample $\{\bX_i:i\in\calS\}$
        \STATE\textbf{let} $\hat{\bSigma}_{\calS}$ be the
        covariance matrix of the filtered sample
        $\{\bX_i:i\in\calS\}$
        \IF{$p > 1 $ }
            \STATE\textbf{let} $V$ be the span of the top
            $\lceil p/e\rceil$ principal components of
            $\hat{\bSigma}_{\calS} -\bSigma$
            \STATE\textbf{let} $\Proj_{V}$ be the
            orth.~projection onto $V$
            \STATE\textbf{let} $\Proj_{V^\perp}$ be the
            orth.~projection onto the orth.~complement of $V$
            \STATE\textbf{let} $\hat{\bmu} \gets  \Proj_{V^\perp}
            \barx_{\calS} + \textsf{SDR} (\Proj_{V}\bX_1,\ldots,
            \Proj_{V}\bX_n;\Proj_{V}\bSigma\,\Proj_{V}, t)$
        \ELSE
            \STATE\textbf{let}  $\hat\bmu \gets  \med$
        \ENDIF
        \RETURN $\hat{\bmu}$
    \end{algorithmic}
    \end{algorithm}
\end{center}

\def\sfP{\mathsf P}
\section{Adversarially corrupted sub-Gaussian model
and spectral dimension reduction}\label{sec:2}

We assume that a set $\bX_1, \ldots,\bX_n$ of $n$ data
points drawn from a distribution $\sfP_n$ is given.
This set is assumed to contain at least $n-[n\varepsilon]$
inliers, the remaining points being outliers. All the
points lie in the $p$-dimensional Euclidean space and
the inliers are independently drawn from a reference
distribution, assumed to be sub-Gaussian with mean
$\bmu^*\in \mathbb R^p$ and covariance matrix $\bSigma$.
To state the assumptions imposed on the observations in
a more precise way, let us recall that the random vector
$\bzeta$ is said to be sub-Gaussian  with zero mean
and identity covariance matrix, if $\mathbb E[\bzeta] =
0 $, $\mathbb E[\bzeta\bzeta^\top] = \bfI_p$ and for
some $\mathfrak s > 0$, we  have
\begin{align}
    \mathbb{E}\big[e^{\bv^\top \bzeta}\big] \le
    \exp\big\{{\mathfrak s}\|\bv \|^2/2\big\},
    \quad \forall \bv \in \RR^p.
\end{align}
The parameter $\mathfrak s$ is commonly called the
variance proxy and the writing  $\bzeta \sim
{\textup{SG}}_p(\mathfrak s)$ is used.

\begin{definition} We say that the data generating
distribution $\sfP_n$ is an adversarially corrupted
sub-Gaussian distribution with mean $\bmu^*$, covariance
matrix $\bSigma$, variance proxy $\mathfrak s$ and
contamination rate $\varepsilon$, if there is a probability
space on which we can define a sequence of random vectors
$(\bX_1,\bY_1),\ldots,(\bX_n,\bY_n)$ such that
\vspace{-10pt}
\begin{enumerate}
	\item $\bY_1,\ldots,\bY_n$ are independent and  $ \bSigma^{-1/2}(\bY_i-\bmu^*) \sim
	\textup{SG}_p(\mathfrak s)$ for every $i\in[n]$.
	\item the cardinality of $\mathcal O = \{i\in[n] :
	\bY_i\neq\bX_i\}$ is at most equal to $n\varepsilon$.
	\item the distribution of $(\bX_1,\ldots,\bX_n)$ is $\sfP_n$.
\end{enumerate}
\vspace{-10pt}
We write then\footnote{SGAC stands for sub-Gaussian with
adversarial contamination.} $\sfP_n\in \textup{SGAC}(
\bmu^*,\bSigma,\mathfrak s,\varepsilon)$. In the particular
case where all $\bY_i$ are Gaussian, we will write
 $\sfP_n\in\textup{GAC}(\bmu^*,\bSigma,\varepsilon)$.
\end{definition}

For an overview of various kind of contamination
models we refer the interested reader to \citep{bateni2020}.
The adversarial contamination considered throughout this
work is perhaps the most general one considered in the
literature as the elements of the set $\calO$---called
outliers---may be chosen using $\bmu^*,\bSigma,\varepsilon$
but also $\bY_1,\dots,\bY_n$ by an omniscient adversary.
Note that in this setting, even the set $\calO$ is random
and depends on $\bY_1,\ldots,\bY_n$. Therefore, the inliers
$\{\bX_i:i\in\calI\}$ cannot be considered as independent
random variables. The problem studied in this work consists
in estimating the mean $\bmu^*$ of the reference
distribution from the adversarially corrupted observations
$\bX_1, \ldots,\bX_n$.

The estimator we analyze in this work is termed iterative
spectral dimension reduction and denoted by $\bmuSDR$. It
is closely related to the agnostic mean \citep{LaiRV16} and
to iterative filtering \citep{DiakonikolasKK017}
estimators. We will prove that SDR enjoys most of
desired properties in the setting of robust estimation
of the sub-Gaussian mean.

The parameters given as input to the iterative spectral
dimension reduction algorithm are a strictly decreasing
sequence of positive integers $p_0,\ldots,p_L$ such that
$p_0 = p$ and a positive threshold $t > 0 $. We recall
that the geometric median is defined by
\begin{align}\label{GM}
    \bmuGM\in\text{arg}\min_{\bmu\in\mathbb R^p}
    \sum_{i=1}^n \|\bX_i-\bmu\|_2.
\end{align}
The algorithm for computing the SDR estimator reads as
follows.
\begin{enumerate}\label{alg:3}
    \item Start by setting $\bfV_0 = \bfI_p$.
    \item For $\ell = 0,\ldots L-1$ do

    \begin{enumerate}

        \item Define $\bar\bmu^{(\ell)}\in\mathbb R^{p_\ell}$
        as the geometric median of  $\{\bfV_\ell^\top\bX_i:
        i\in[n]\}$.

        \item Define the set $\calS^{(\ell)} = \big\{i\in[n]:
        \|\bfV_\ell^\top\bX_i - \bar\bmu^{(\ell)}\|_2 \leq
        t\sqrt{p_\ell}\big\}$ of filtered data points.

        \item Let $\barx^{(\ell)}$ and $\hat\bSigma{}^{(\ell)}$
        be the mean vector and the covariance matrix of the
        filtered sample $\{\bX_i : i\in\calS^{(\ell)}\}$,
        that is
        \begin{align}
            \barx^{(\ell)} = \frac1{|\calS^{(\ell)}|}
            \sum_{i\in\calS^{(\ell)}}\bX_i,\qquad
            \hat\bSigma{}^{(\ell)} = \frac1{|\calS^{(\ell)}|}
            \sum_{i\in\calS^{(\ell)}}
            (\bX_i - \barx)^{\otimes 2}.
        \end{align}

    \item Set $\hat\bmu^{(\ell)} = \bfV_\ell \bfU_\ell^\top
     \bfU_\ell\bfV_\ell^\top\barx^{(\ell)}$, where $\bfU_\ell$
    is a $(p_\ell - {p_{\ell+1}})\times p_\ell$ orthogonal
    matrix the rows of which are the eigenvectors of
    $\bfV_\ell^\top(\hat\bSigma{}^{(\ell)} - \bSigma)\bfV_\ell$
        corresponding to its $(p_\ell - {p_{\ell+1}})$ smallest
    eigenvalues.


    \item Set $\bfV_{\ell+1} = \bfV_\ell(\bfU_\ell^{\perp}
    )^\top\in\mathbb R^{p\times p_{\ell+1}}$, where
    $\bfU_\ell^{\perp}$ is a $p_{\ell+1} \times p_\ell$
    orthogonal matrix orthogonal to $\bfU_\ell$, that is
    $\bfU_\ell^{\perp}\bfU_\ell^\top =\mathbf 0$.
\end{enumerate}


\item Define $\bar\bmu^{(L)}$ as the geometric
    median of  $\bfV_L^\top\bX_i$ for $i=1,\ldots,n$ and set
    $\calS^{(L)} = \big\{i\in[n]: \|\bfV_L^\top\bX_i -
    \bar\bmu^{(L)}\|_2 \leq t \sqrt{p_L}\big\}$.
\item Define $\hat\bmu^{(L)} = \bfV_L\bfV_L^\top\,\barx^{(L)}$,
the average of filtered and projected vectors.
\item Return $\bmuSDR = \hat\bmu^{(0)} + \hat\bmu^{(1)} +
    \ldots + \hat\bmu^{(L)}$.
\end{enumerate}

The steps described above can be summarised as follows.
At each iteration $\ell<L$, we start by determining a
filtered subsample $\calS^{(\ell)}$ and a
``nearly-outlier-orthogonal'' subspace $\mathscr U_\ell =
{\rm Im}(\bfV_\ell\bfU_\ell^\top)$ of $\mathbb R^p$ of
dimension $p_{\ell} - p_{\ell+1}$.  We define the projection
of $\bmuSDR$ onto $\mathscr U_\ell$ as the sample mean of
the filtered and projected subsample, and we move to the
next step for determining the projection of $\bmuSDR$ onto
the remaining part of the space. At the last iteration
$L$, when the dimension is well reduced, the projection of
$\bmuSDR$ onto the subspace $\mathscr U_L$ is defined as
the average of the filtered subsample projected onto
$\mathscr U_L$. The subspaces $\mathscr U_\ell$ are two-by-two
orthogonal and span the whole space $\mathbb R^p$.
Each subspace is determined from the spectral decomposition
of the covariance matrix of the data points projected
onto $(\mathscr U_0\oplus \ldots\oplus\mathscr U_{\ell-1}
)^\perp$, after removing the points lying at an abnormally
large distance from the geometric median.

\subsection{Choice of the dimension reduction regime}
The analysis of the error of the SDR estimator conducted
in this work leads to an upper bound in which the sequence
$(p_0,\ldots,p_L)$ is involved only through the expression
\begin{align}
    F(p_0,\ldots,p_L)
        = \sum_{\ell=1}^{L}  \frac{p_{\ell-1}}{p_\ell}.
\end{align}
Therefore, an appealing way of choosing this sequence is
to minimize the function $F$ under the constraint that the
sequence is decreasing and $p_0=p$ and $p_L=1$. It follows
from the inequality between the arithmetic and geometric
means that $F(p_0,\ldots,p_L)\ge L p^{1/L}$. Furthermore,
the equality is achieved\footnote{We relax here the assumption
that all the entries $p_\ell$ are integers.} in the case when
all the terms in the definition of $F$ are equal,
\textit{i.e.}, when for some $c>0$ we have $p_{\ell-1} = c
p_\ell$ for every $\ell \in[L]$. Since $p_0 = p$ and $p_L=1$,
this yields $c = p^{1/L}$ or, equivalently, $L = \log p/\log c$.
Using these relations, we find that the function $F$ is lower
bounded by $Lc = (c/\log c)\log p$. The last step is to find
the minimum of the function $c\mapsto c/\log c$ over the
interval $(1,\infty)$. One easily checks that this function has
a unique minimum at $c = e$. All these considerations
advocate for using the dimension reduction regime defined by
\begin{align}\label{def:p}
    p_0 = p,\quad p_{\ell} = \lfloor p_{\ell-1}/e\rfloor +1,
    \ \ell \in[L] ,\quad p_L = 1,
\end{align}
where $\lfloor x\rfloor$ is the largest integer strictly
smaller than $x$. Such a definition of $(p_\ell)$ ensures
that $p_{\ell-1}/p_\ell\le e$ and that\footnote{To check
this inequality, one can use the fact that $3\le p_{L-2}\le
pe^{2-L} + e/(e-1)$. This implies $L \le 2\log p$ for
$p\ge 6$. For smaller values of $p$, the inequality can
be checked by direct computations.} $L \le 2\log p$. In the
rest of the paper, we assume that the sequence $(p_\ell)$ is
chosen as in \eqref{def:p}.

\subsection{Choice of the threshold}
The SDR procedure has one important tuning parameter:
the threshold $t$ used to discard clearly outlying
data points. Let us introduce the auxiliary notation
\begin{align}\label{bar_rn}
    \bar{\sf r}_n = \frac{\sqrt{\rSigma} +
    \sqrt{2 \log( 2/\delta)}}{\sqrt{n}},\qquad
    \text{and}\qquad
    \tau = \frac14\bigwedge \frac{\bar{\sf r}_n}{\sqrt{
    \log(2/\bar{\sf r}_n)}}.
\end{align}
Note that $\bar{\sf r}_n$ is essentially the quantile,
up to a universal constant factor, of order $1-\delta$
of the distribution of $\|\bar\bY_n-\bmu^*\|_2$ where
$\bY_i$'s are independently drawn from $\mathcal N_p
(\bmu^*,\bSigma)$ with $\|\bSigma\|_{\rm op}=1$. Our
theoretical results advocate for using the value
$t = t_1 + t_2$, where
\begin{align}
    t_1 = \frac{2(1 + \bar{\sf r}_n)}{1-2\varepsilon^*},
    \qquad
    t_2 = 1 + \frac{\bar{\sf r}_n}{\sqrt{\tau}} +
    \sqrt{2 + \log(2/{\tau})},
\end{align}
where $\varepsilon^*<1/2$ is the largest value of the
contamination rate that the algorithm may handle.

Let $\bxi_1,\ldots,\bxi_n$ be independent Gaussian with zero mean and covariance $\bSigma$. The expression of $t_1$ is obtained as an upper
bound on the quantile of order $1-\delta/2$ of the
distribution of the random variable
\begin{align}
    T_1 = \sup_{V}\frac{2}{n(1-2\varepsilon)\dim(V)}
    \sum_{i=1}^n\|\ProjV\bxi_i\|_2,
\end{align}
see \Cref{lem:geommed} and its proof for further details.
Similarly, $t_2$ is defined so that the
event
\begin{align}
	\sup_{V} \sum_{i=1}^n \indic\big( \|\ProjV\bxi_i\|_2^2
	> t_2^2\dim(V) \big) \le n\tau
\end{align}
has a probability at lest $1-\delta/2$. The related
computations are deferred to \Cref{ssec:7.2}. Although we
tried to get sharp values for these thresholds $t_1$ and
$t_2$, it is certainly possible to improve these values
either by better mathematical arguments or by empirical
considerations. Of course, smaller values of the thresholds
$t_1$ and $t_2$ satisfying aforementioned conditions lead
to an SDR estimator having smaller error.

\section{Assessing the error of the SDR estimator}
\label{sec:3}
The iterative spectral dimension reduction estimator
defined in previous sections has some desirable properties
of a robust estimator that are easy to check. In particular,
it is clearly equivariant by translation, orthogonal linear
transform and global scaling. Furthermore, the breakdown
point of the estimator is equal to that of the geometric
median, that is to $1/2$. This means that even if almost
the half of data points are chosen to be infinitely large,
the estimator will not ``break down'' in the sense of
becoming infinitely large. However, the fact that the
estimated value does not become infinitely large, it might
be not very close to the true mean. The next theorem
shows that this is not the case and that the error of
the SDR estimator has a nearly rate-optimal behavior
even when the contamination rate is close to $1/2$. The
adverb ``nearly'' is used here to reflect the presence of
the $\sqrt{\log p}$ factor in the error bound, which is
not present in the minimax rate.

\begin{theorem}\label{thm:1}
Let $\varepsilon^* \in (0, 1/2)$, and $\delta \in (0, 1/2)$.
Define $\bar{\sf r}_n$ and $\tau$ as in \eqref{bar_rn}.
For every $\varepsilon \le \varepsilon^*$, let $\bmuSDR$
be the estimator returned by Algorithm~\ref{alg:3}
with
\begin{align}
    t=\frac{3-2\varepsilon^*}{1-2\varepsilon^*}
      \bigg(1 + \frac{\bar{\sf r}_n}{\sqrt{\tau}}\bigg)
      +\sqrt{2 + 2\log\big(1/{\tau}\big)}.
\end{align}
There exists a universal constant $\sfC$ such that for
every $\sfP_n\in \textup{GAC}(\bmu^*,\bSigma,\varepsilon)$
with $\varepsilon\le \varepsilon^*$ and\footnote{Since in
this theorem $\bSigma$ is assumed to be known, we can always
divide all the data points $\bX_i$ by $\|\bSigma\|_{\rm op
}^{1/2}$ to get a data set with a covariance matrix
satisfying $\|\bSigma\|_{\rm op}=1$.} $\|\bSigma\|_{\rm op}
=1$, the probability of the event
\begin{align}
    \big\|\bmuSDR - \bmu^*\big\|_2
    &\le \frac{\sfC\,\sqrt{\log p}}{1-2\varepsilon^*}
    \bigg(\sqrt{\frac{\rSigma}{n}}+ \varepsilon\sqrt{
    \log(2/\varepsilon)} + \sqrt{\frac{\log(1/\delta)}{n}}
    \bigg)
\end{align}
is at least $1-\delta$. Moreover, the constant $\sfC$
from the last display can be made explicit by replacing
the effective rank $\rSigma$ by the dimension $p$ in
the definition of $\bar{\sf r}_n$: That is, for every
$\delta \in (0, 1/5)$ the inequality
\begin{align}
    \big\|\bmuSDR - \bmu^*\big\|_2
    &\le \frac{156\, \sqrt{2\log p}}{1-2\varepsilon^*}
    \bigg( \sqrt{\frac{2p}{n}}+ \varepsilon\sqrt{
    \log(2/\varepsilon)} + \sqrt{\frac{3\log(2/\delta)}{n}}
    \bigg)
\end{align}
holds with probability at least $1 - 5 \delta$.
\end{theorem}

If we compare this result with its counterpart
established in \citep{dalalyan2020allinone} for the
iteratively reweighted mean, besides the extra $\log p$
factor, we see that the above error bound does not
reduce to the error of the empirical mean when the
contamination rate goes to zero. We do not know whether
this is just a drawback of our proof, or it is an
intrinsic property of the estimator. Our numerical
experiments reported later on suggest that is might
be a property of the estimator.

There is another logarithmic factor, $\sqrt{\log(2/
\varepsilon)}$, present  in the second term of the
error bounds provided by the last theorem, which does
not appear in the minimax rate. There are computationally
intractable robust estimators of the Gaussian mean, such
as the Tukey median, that have an error bound free of
this factor. However, all the known error bounds provably
valid for polynomial time algorithms has this extra
$\sqrt{\log(2/\varepsilon)}$ factor. Furthermore,
this factor is known to be unavoidable in the case of
sub-Gaussian model with adversarial
contamination\footnote{ Both sub-Gaussianity of the
reference distribution and the adversarial nature of
the contamination are important for getting the extra
$\sqrt{\log(2/\varepsilon)}$ factor in the minimax rate.},
see \citep[Section 2]{LugMend21}.

As shows the next theorem, the claims of \Cref{thm:1}
carry over the sub-Gaussian reference distributions
with some slight modifications. These modifications mainly
stem from the following lemma assessing the tail behavior
of the singular values of a matrix having independent
and sub-Gaussian columns.

\begin{lemma}[\cite{Vershynin2012IntroductionTT}, Theorem 5.39]\label{lem:vershynin:bis}
	\it{
	Let $\bxi_{1:n}$ be a matrix consisting of sub-Gaussian vectors with variance proxy $\mathfrak{s}$. There is a universal constant $\sfC_0$ such that for every $t>0$ and for every pair of
	positive integers $n$ and $p$,
	we have
	\begin{align}
		\bfP\big(s_{\min}(\bxi_{1:n})  \le \sqrt{n} - \sfC_0{\mathfrak s}(\sqrt{p}
		+ t)\big)  \le e^{-t^2},\\
		\bfP\big(s_{\max}(\bxi_{1:n})  \ge \sqrt{n} + \sfC_0{\mathfrak s}(\sqrt{p}
		+ t)\big)  \le e^{-t^2}.
	\end{align}
	}
\end{lemma}

Note that in the Gaussian case $\mathfrak s =1$ and the
constant $\sfC_0$ can be chosen equal to $\sqrt{2}$. The last lemma
leads to the following adaptations in the values of the
thresholds used in the SDR estimator. First, we introduce auxiliary definitions
\begin{align}\label{bar_rn_subG}
    \tau = \frac14\bigwedge \frac{\bar{\sf r}_{n,\mathfrak{s}}}{\sqrt{\log_+(2/\bar{\sf r}_{n,\mathfrak{s}})}}, \qquad\text{with}\qquad
    \bar{\sf r}_{n,\mathfrak{s}} = \frac{3\sqrt{\mathfrak{s}}\big(\sqrt{p} + 2\sqrt{\log(2/\delta)}\big)}{\sqrt{n}}.
\end{align}
Then, we set $t = t_1 + t_2$ with
\begin{align}
    t_1 = \frac{2(1 + \sfC_0 \, \bar{\sf r}_{n,\mathfrak{s}}\sqrt{\mathfrak{s}})}{1-2\varepsilon^*},
    \qquad
    t_2 = 1 + \sfC_0 \sqrt{{\mathfrak{s}}}\Big(\frac{\bar{\sf r}_{n, \mathfrak{s}}}{\sqrt{\tau}} +
    \sqrt{2 + \log(1/{\tau})}\Big),
\end{align}
where  $\sfC_0$ is the same as in \Cref{lem:vershynin:bis}.

Now we are ready to state the theorem for the for sub-Gaussian distributions showing that the SDR estimator with the threshold depending on the variance proxy $\mathfrak{s}$ yields the same upper bound on $\ell_2$ distance between our estimator $\bmuSDR$ and the true value $\bmu^*$ replacing the effective rank $\rSigma$ with the dimension $p$.

\begin{theorem}[Sub-Gaussian version]\label{thm:subG}
Let $\varepsilon^* \in (0, 1/2)$, and $\delta \in (0, 1/2)$.
Define $\bar{\sf r}_{n, \mathfrak{s}}$ and $\tau$ as in \eqref{bar_rn_subG}.
For every $\varepsilon \le \varepsilon^*$, let $\bmuSDR$
be the estimator returned by Algorithm~\ref{alg:3}
with
\begin{align}
    t=\frac{3-2\varepsilon^*}{1-2\varepsilon^*}
      \bigg(1 + \sfC_0 \, \bar{\sf r}_{n,\mathfrak{s}} \sqrt{\frac{\mathfrak{s}}{{\tau}}}\bigg)
      + \sfC_0 \, {\mathfrak{s}} \sqrt{2 + 2\log\big(1/{\tau}\big)},
\end{align}
where $\sfC$ is a universal constant.
Then, there exists a constant ${\sfC}_{\mathfrak{s}}$ depending only on the variance proxy $\mathfrak{s}$ such that for
every $\sfP_n\in \textup{SGAC}(\bmu^*,\bSigma,\mathfrak{s}, \varepsilon)$
with $\varepsilon\le \varepsilon^*$ and $\|\bSigma\|_{\rm op}
=1$, the probability of the event
\begin{align}
    \big\|\bmuSDR - \bmu^*\big\|_2
    &\le \frac{{\sfC}_{\mathfrak{s}}\,\sqrt{\log p}}{1-2\varepsilon^*}
    \bigg(\sqrt{\frac{p}{n}}+ \varepsilon\sqrt{
    \log(2/\varepsilon)} + \sqrt{\frac{\log(1/\delta)}{n}}
    \bigg)
\end{align}
is at least $1-\delta$.
\end{theorem}

\section{The case of unknown covariance matrix}
\label{sec:4}
The SDR estimator, as defined in \Cref{alg:1}, requires the
knowledge of covariance matrix $\bSigma$. In this section we
consider the case where the matrix $\bSigma$ is unknown,
but an approximation of the latter is available. Namely,
we assume that we have access to a matrix $\btSigma$ and to
a real number $\gamma>0$ such that  $\|\btSigma - \bSigma
\|_{\textup{op}} \le \gamma \|\bSigma\|_{\textup{op}}$. In
such a situation, we can replace in the SDR estimator the
true covariance matrix by its approximation $\btSigma$.
This will necessarily require to adjust the threshold $t$
accordingly. The goal of the present section is to propose
a suitable choice of $t$ and to show the impact of the
approximation error $\gamma$ on the estimation accuracy.

\def\tilde{\widetilde}
As mentioned, the parameter $t$ used in \Cref{alg:1}
needs to be properly tuned in order to account for the
approximation error in the covariance matrix. To this
end, we introduce the following auxiliary notation similar
to those presented in \eqref{bar_rn}:
\begin{align}\label{tilde_rn}
    \tilde{\sf r}_n = \frac{\sqrt{\sfC_\gamma\rtSigma} +
    \sqrt{2 \log( 2/\delta)}}{\sqrt{n}}\qquad
    \text{and}\qquad
    \tilde\tau = \frac14\bigwedge \frac{\tilde{\sf r}_n}{\sqrt{
    \log(2/\tilde{\sf r}_n)}},
\end{align}
where $\sfC_\gamma = (1+\gamma)/(1-\gamma)$. Compared to
\eqref{bar_rn}, the main difference here is the presence
of the factor $\sfC_\gamma$ (which is equal to one if
$\gamma=0$) and the substitution of the effective rank of
$\bSigma$ by that of its approximation $\btSigma$. In the rest
of this section, we assume that $\bSigma$ is invertible.

\begin{theorem}\label{thm:3}
Let $\varepsilon^* \in (0, 1/2)$, $\delta \in (0, 1/2)$ and
define $\tilde{\sf r}_n$ and $\tau$ as in \eqref{tilde_rn}.
Assume that $\btSigma$ satisfies $\|\bSigma^{-1/2}\btSigma
\bSigma^{-1/2} - \bfI_p \|_{\textup{op}} \le \gamma$ for
some $\gamma \in (0, 1/2]$. Let $\bmuSDR$ be the output of
\textup{\sf SDR}$(\bX_1,\ldots, \bX_n;\btSigma,\tilde t_\gamma)$,
see Algorithm~\ref{alg:1}, with
\begin{align}
    \tilde{t}_\gamma=\frac{\|\btSigma\|_{\rm op}}{1-\gamma}
        \bigg\{
        \frac{3-2\varepsilon^*}{1-2\varepsilon^*}
        \bigg(1 + \frac{\tilde{\sf r}_n}{\sqrt{\tilde\tau}}
        \bigg) +\sqrt{2 + \log\big(2/{\tau}\big)}\bigg\}.
\end{align}
Then, there exists a universal constant $\sfC$ such that for
every data generating distribution $\sfP_n\in \textup{GAC}(
\bmu^*,\bSigma,\varepsilon)$ with $\varepsilon\le \varepsilon^*$,
the probability of the event
\begin{align}\label{eq:th3}
    \big\|\bmuSDR - \bmu^*\big\|_2
    &\le \frac{\sfC\, \|\bSigma\|^{1/2}_{\textup{op}}
    \sqrt{\log p}}{1-2\varepsilon^*} \bigg(\sqrt{\frac{\rSigma}{n}}
    + \varepsilon\sqrt{\log(2/\varepsilon)} + \sqrt{\varepsilon\gamma}
    + \sqrt{\frac{\log(1/\delta)}{n}} \bigg)
\end{align}
is at least $1-\delta$.
\end{theorem}

On the one hand, if the value of $\gamma$ is at most of order
$\sqrt{(\rSigma/n) \log(1/\varepsilon)}  + \varepsilon
\log(1/\varepsilon)$ then \Cref{thm:3} implies that the
estimation error is of the same order as in the case of
known covariance matrix $\bSigma$ (\Cref{thm:1}). For
instance, if the matrix $\bSigma$ is assumed to be diagonal,
one can defined $\btSigma$ as the diagonal matrix composed
of robust estimators of the variances of univariate
contaminated Gaussian samples; see, for instance, Section
2 in \citep{Comminges21}. For recent advances on robust estimation
of (non-diagonal) covariance matrices by computationally
tractable algorithms we refer the reader to \citep{D0W19}.

On the other hand, if the value of $\gamma$ for which the
condition $\|\btSigma - \bSigma \|_{\textup{op}} \le \gamma
\|\bSigma\|_{\textup{op}}$ is known to be true is of
larger order than $\sqrt{(\rSigma/n) \log(1/\varepsilon)}
+ \varepsilon \log(1/\varepsilon)$, then $\sqrt{\varepsilon\gamma}$
dominates the other terms appearing in the error bound
\eqref{eq:th3}. Moreover, if $\gamma$ is of constant order,
then we get the error rate $\sqrt{\frac{\rSigma}{n}} +
\sqrt{\varepsilon}$, which is in line with previously known
bounds for computationally tractable estimators; see for example
\citep[Tjheorem 1.1]{LaiRV16}, \citep[Theorem 3.2]{DiakonikolasKK017},
\citep[Theorem 4]{dalalyan2020allinone}.

\section{Numerical experiments}
\label{sec:5}
We conducted numerical experiments on synthetic contaminated
data to corroborate our theoretical results. The main goal
of these experiments is to display statistical and computational
features of the SDR and their dependence on various parameters.
Moreover, we compared SDR to some other estimators proposed in
the literature as well as to the oracle (empirical mean of the inliers). To do so, we selected
componentwise median (CM), geometric median (GM) and Tukey's
median (TM) as the three classic estimators of the context,
and the iteratively reweighted mean (IRM), introduced in
\citep{dalalyan2020allinone}, as an example of optimization
based method.

\subsection{Implementation details}
The experiments were run on a laptop with a 1.8 GHz Intel
Core i7 and 8 GB of RAM. R codes of the experiments
are freely available on the last author’s website. For GM and
TM the R packages
Gmedian\footnote{https://cran.r-project.org/package=Gmedian} (\cite{Gmedian})
and TukeyRegion\footnote{https://cran.r-project.org/web/packages/TukeyRegion/index.html} (\cite{TM})
were used. IRM had been already implemented in R using
Mosek\footnote{www.mosek.com}.

To optimize SDR, several choices were made. First,
since geometric median is used in SDR as a rough estimator
of the location, we limited it to at most 15 iterations and
to stop at an accuracy of order 1. See the reference
manual of Gmedian to have more details on these parameters.
Second, since at the last step of the SDR one can use any
estimator which is robust in low-dimensional setting, we
chose to use the median of the projected data points. Finally,
we adjusted the numerical constant in the threshold $t$.
In all the experiments, we assumed that the true value of
$\varepsilon$ is known and used $\varepsilon^*=\varepsilon$.

\subsection{Experimental setup}

Experiments were conducted on synthetic data sets obtained
by applying a contamination scheme to $n$ i.i.d.\ samples
drawn from $\mathcal{N}_p(\mathbf 0,\bfI_p)$. The following
contamination schemes were considered.
\begin{itemize}
    \item \textit{Contamination by uniform outliers} (CUO):
    the locations of $n\varepsilon$ outliers are chosen
    at random independently of the inliers. The outliers are
    independent Gaussian with identity covariance matrix and
    with means having coordinates independently drawn from the
    uniform in $[0,3]$ distribution.
    \item \textit{Gaussian mixture contamination} (GMC):
    the locations of $n\varepsilon$ outliers are chosen
    at random independently of the inliers. The outliers are
    independent $\mathcal{N}_p(\bmu,\bfI_p)$. In our experiments,
    we chose $\bmu$ such that $\|\bmu\|=15$.
    \item \textit{Contamination by "smallest" eigenvector} (CSE):
    We replace the $n\varepsilon$ samples most correlated with
    the smallest principal eigenvector $\boldsymbol v_p$ of
    the sample covariance matrix, by $n\varepsilon$ vectors
    all equal to $\sqrt{p}\,\boldsymbol v_p$ ($\boldsymbol v_p$
    is assumed to be a unit vector). In contrast with the two
    previous schemes, this one is adversarial.
\end{itemize}
Each experiment was repeated 50 times for SDR, CM, GM, the oracle and 10
times for IRM and TM. The tolerance probability $\delta$ was
set to 0.1 in all the experiments. In the figures, points on
the curves are median values of the error or of the running time
for these trials whereas vertical bars overlaid on the points
show the spread between the first and third quartiles. Since
the computation of TM is prohibitively costly and is possible
only for small sample sizes and dimensions, it is excluded
from most of the experiments.

\subsection{Statistical accuracy}
At the upper left panel of \Cref{fig:1}, we illustrate the
behavior of the risk when the sample size increases for four
different contamination levels: $\varepsilon\in\{0.1,0.2,0.3,0.4\}$.
The data are of dimension 60 and generated by the GMC scheme.
The median estimation error converges respectively to the values
0.18, 0.36, 0.62 and 1.06. According to our theoretical result, the
limit of the error should be proportional to $\frac{\varepsilon \log
(1/\varepsilon)}{1-2\varepsilon}$. This is confirmed by the
experimental results, since the ratio between the empirical limit
of the median error and $\frac{\varepsilon\log(1/\varepsilon)}{
1-2\varepsilon}$ for each $\varepsilon$ is between 0.58 and 0.69.

At the upper right panel of \Cref{fig:1}, the dependence of the
error on the dimension is displayed. To better illustrate the effect
of the dimension on the estimation error, we carried out our
experiment on data sets of small sample size $n=100$ with
CUO contamination. We compared the error of GM, CM, IRM and the
oracle. In this plot, we clearly
observe the supremacy of SDR and IRM as compared to GM and CM,
which is in line with theoretical results. An important observation
is that the error of the SDR estimator is very close to those
of the IRM estimator and the oracle. This suggests that the factor
$\sqrt{\log p}$ present in our theoretical results might be
an artifact of the proof rather than an intrinsic property of
the estimator, at least for nonadversarial contamination.
    \begin{figure}
    \centering
    \includegraphics[width=.45\textwidth]{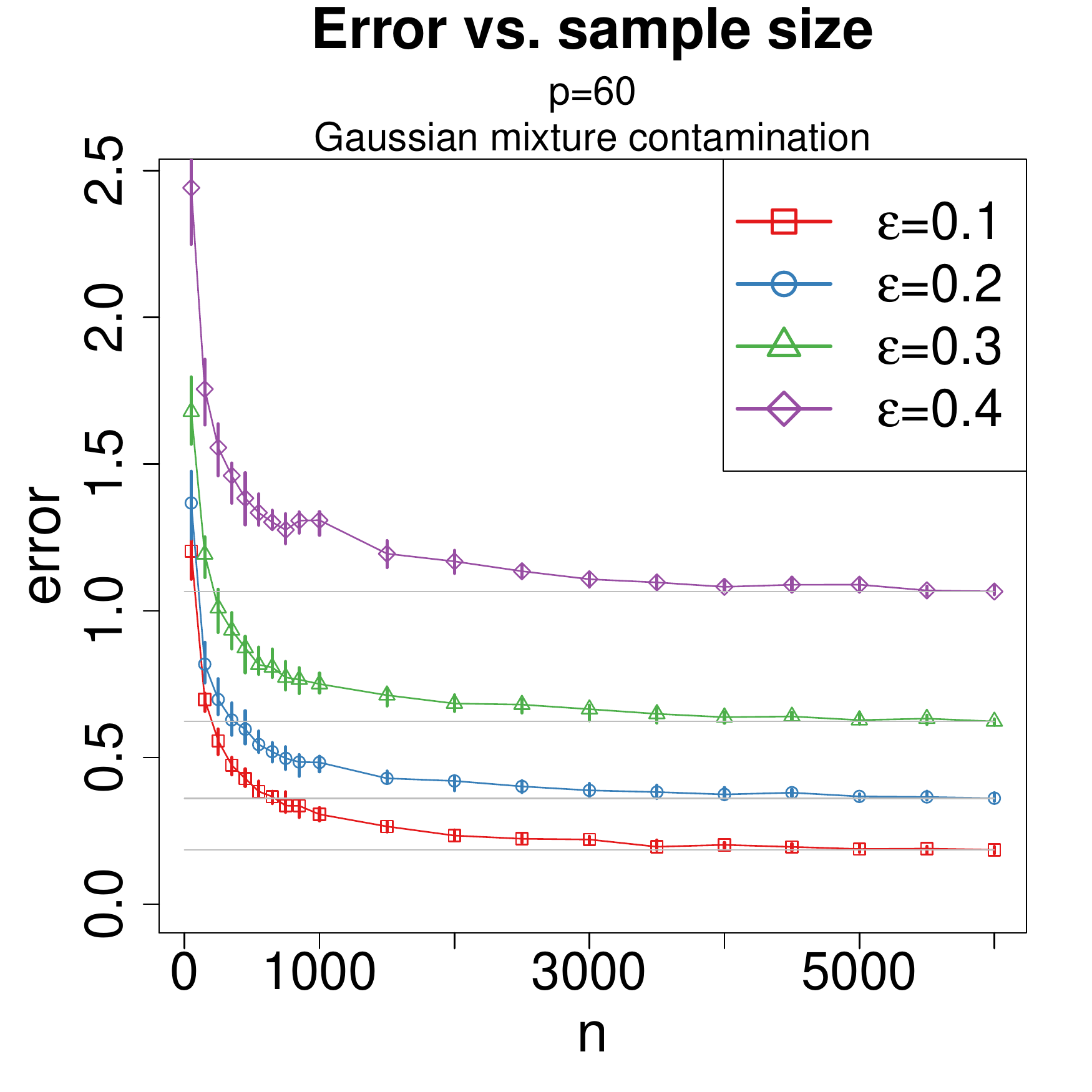}
    \includegraphics[width=.45\textwidth]{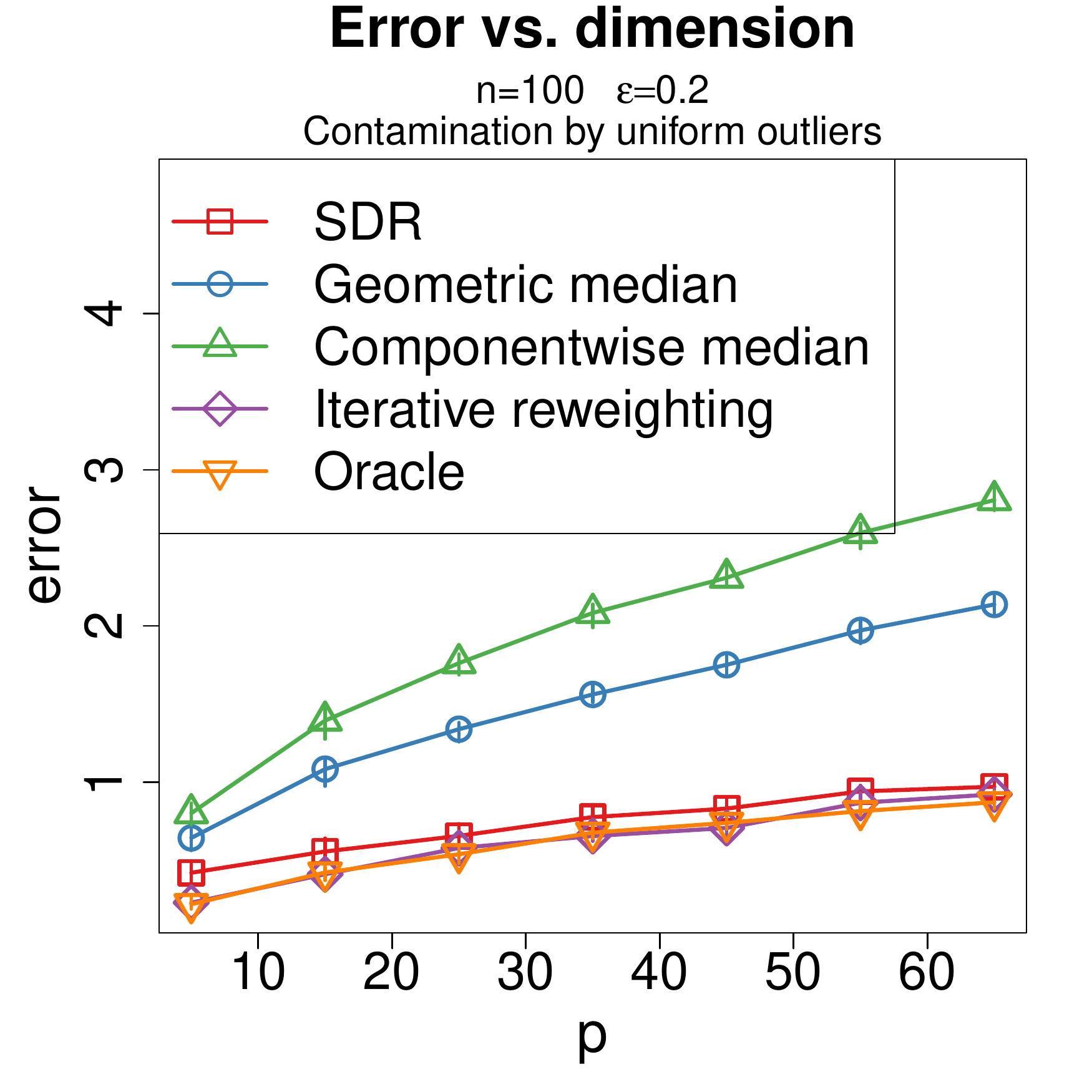}
    \includegraphics[width=.45\textwidth]{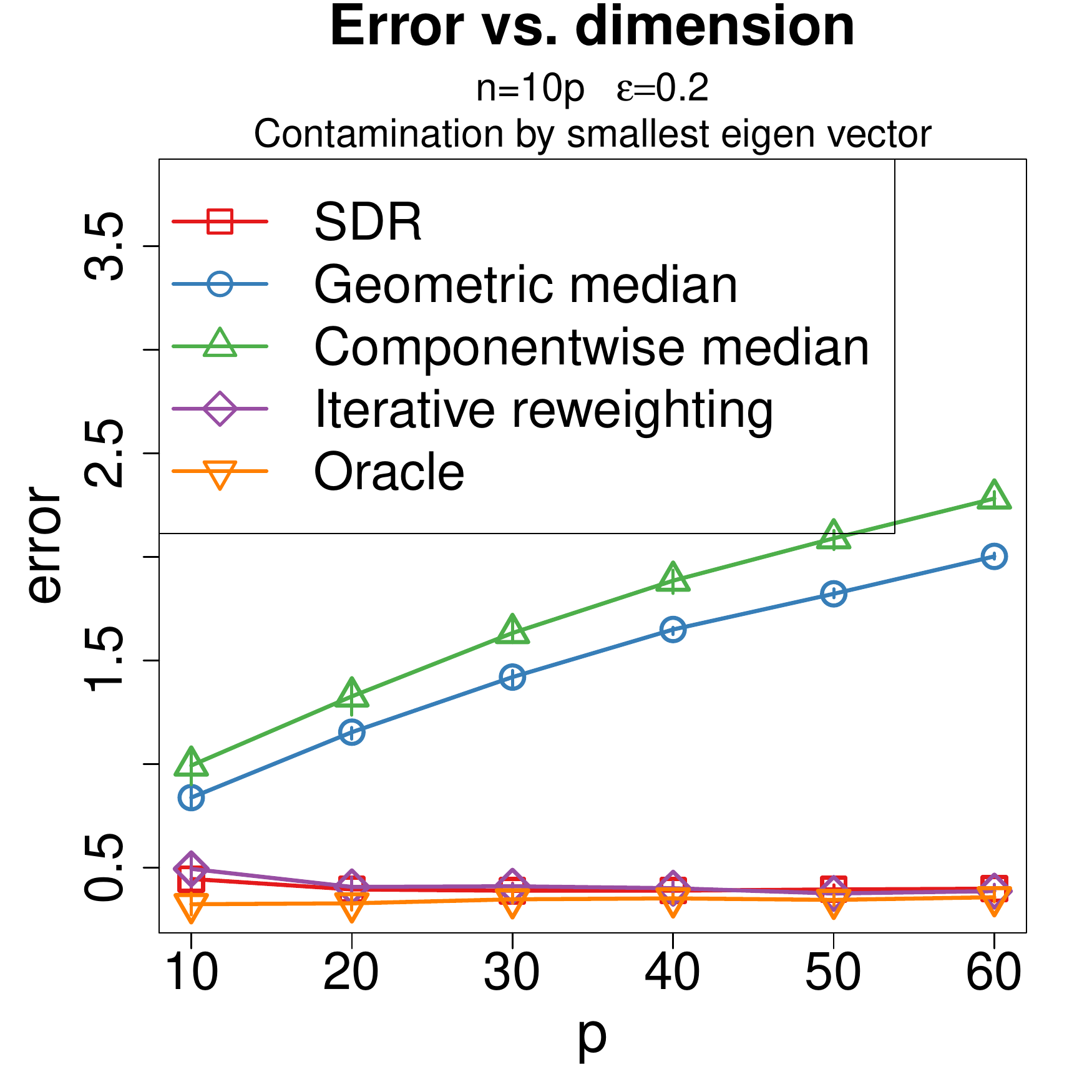}
    \includegraphics[width=.45\textwidth]{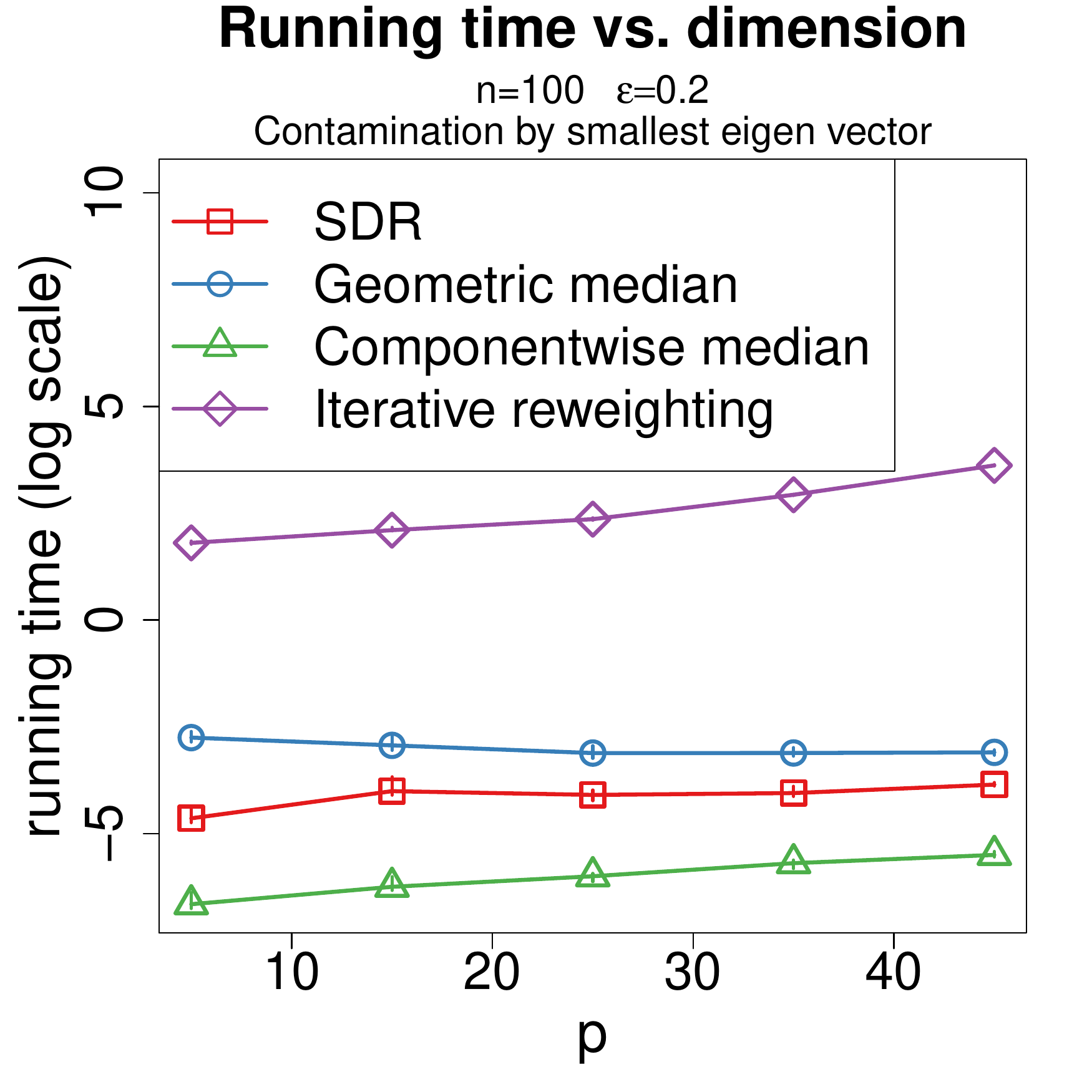}
    \caption{\label{fig:1} The upper left panel illustrates the
    convergence of SDR's median error when the sample size tends
    to infinity, for various contamination rates. The limiting
    values are shown by gray lines. The upper right panel shows
    the effect of dimension on the error. We see that in the
    case of SDR this effect is almost the same as for IRM and
    the oracle. The lower left panel plots the quantities when
    the sample-size increases proportionally to the dimension.
    Once again, we see that SDR is almost as accurate as IRM and
    the oracle. The lower right panel plots the running times of
    different estimators for various dimensions. It shows the huge
    computational gain of the SDR estimator as compared to IRM.
    }
    \end{figure}

The last experiment aiming to display the behavior of the
estimation error is depicted in the lower left panel of \Cref{fig:1}.
The examined synthetic datasets were generated by the CSE scheme with
$\varepsilon=0.2$. We measured the error for different values of
the dimension and for sample size $n=10p$ proportional to the
dimension. In this case, the term $\sqrt{p/n}$ in the risk bound
remains unchanged and we may perceive if the dimension virtually
effects the term dependent on $\varepsilon$ in the bound. The
plot clearly confirms that the error is stable for SDR as it is for
IRM and the oracle, in sharp contrast with GM and CM. The last
point, of course, is not surprising since the risks of GM and CM
scale as  $\varepsilon \sqrt{p}$. Once again, this plot suggests
that the factor $\sqrt{\log p}$ present in the SDR's risk bound
might be unnecessary.

\subsection{Computational efficiency}

We conducted another experiment in order to better understand
the computational complexity of  SDR. Note that the computational
cost of SDR comes from two operations done at each iterations:
SVD of sample covariance matrix and computation of geometric median.
We see that SDR can be computed in a reasonable time even in high
dimensions. For instance, for $n=10000$ and $p=1000$ it takes
nearly 26 seconds (tested over 20 trials).

At the lower right panel of \Cref{fig:1}, we plotted the running
times (in seconds) of GM, CM, IRM and SDR for various dimensions.
Sample size in this experiment was set to 100, contamination rate
was $\varepsilon=0.2$ and CSE contamination scheme was used.
As expected, IRM has substantially larger running time compared to
SDR, GM and CM; this is due to the semidefinite programming solver
running at each iteration of IRM. The fact that SDR is faster than
GM (even though GM is deployed at each iteration of SDR) is
explained by our choice of computing only a rough approximation of
GM within SDR (limiting to 15 iterations and a tolerance parameter
set to 1).

\subsection{Breakdown point}
A natural measure of robustness of an estimator is its resistance
to a large fraction of outliers. The goal here is to demonstrate
empirically our theoretical result showing that the breakdown point
of the SDR estimator is 1/2.

In \Cref{fig:4}, at the left panel, we evaluated the error of the
estimators on samples of size 100 and dimension 10 generated by
the CSE scheme, for various values of $\varepsilon$. We can
observe that SDR preserves its robustness with large contamination
rates and outperforms other estimators, excepted the oracle. More
precisely, SDR and IRM have roughly the same error up to
$\varepsilon=0.28$. Starting from this value, the error of IRM
starts a steep deterioration joining CM and GM.

At the right panel of \Cref{fig:4}, we plotted the error as a
function of the contamination rate for TM, CM, GM, IRM and SDR.
Data used in this experiment were of size 100 and dimension 3,
corrupted by GMC scheme. For this type of contamination, we observe
that the IRM estimator remains robust even for $\varepsilon$ close
to $1/2$, whereas the error of TM deteriorates significantly for
$\varepsilon>1/3$. As a conclusion, for two contamination schemes
which are challenging for iteratively reweighted mean and Tukey's
median, SDR shows very stable behavior.

\begin{figure}
    \centering
    \includegraphics[width=.45\textwidth]{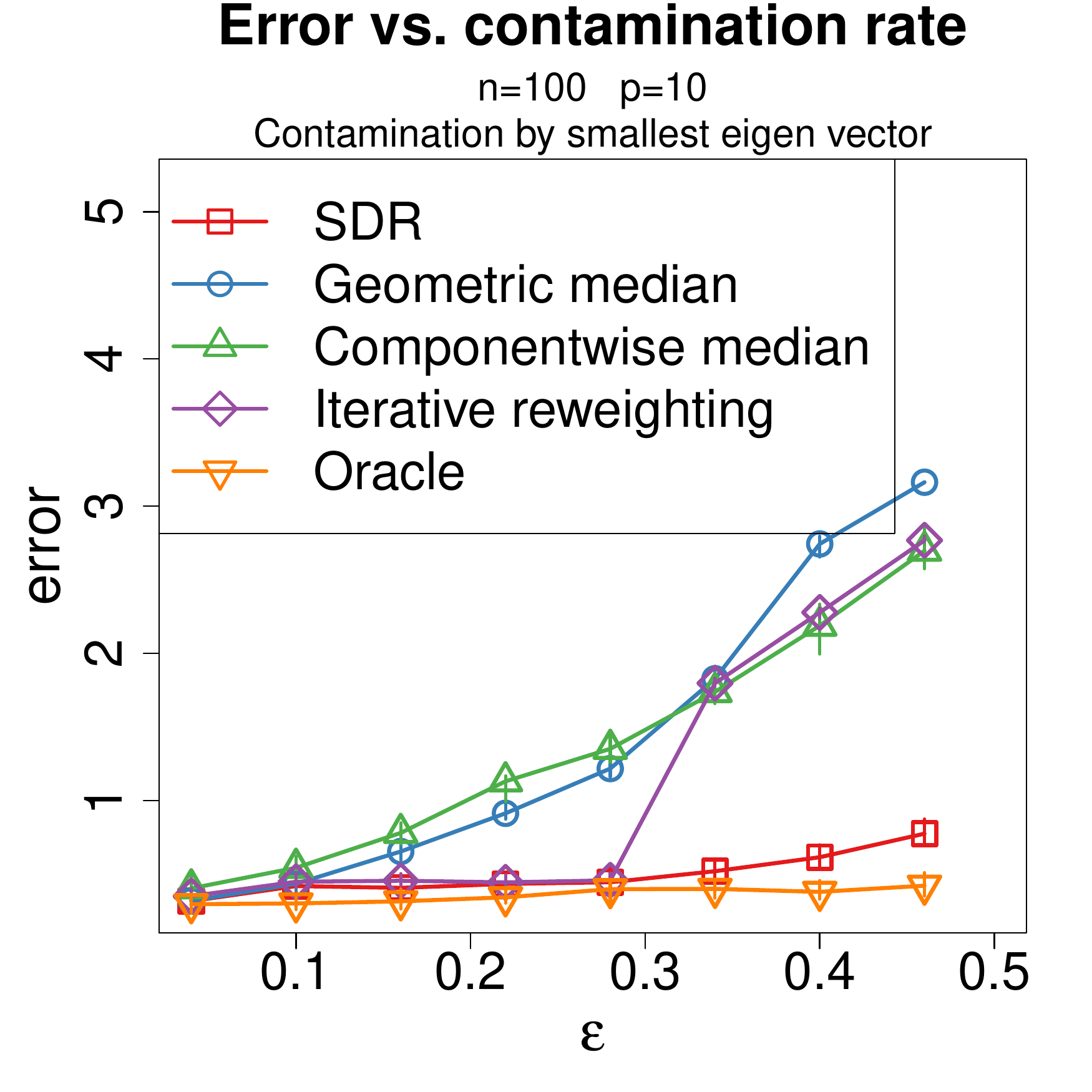}
    \quad\includegraphics[width=.45\textwidth]{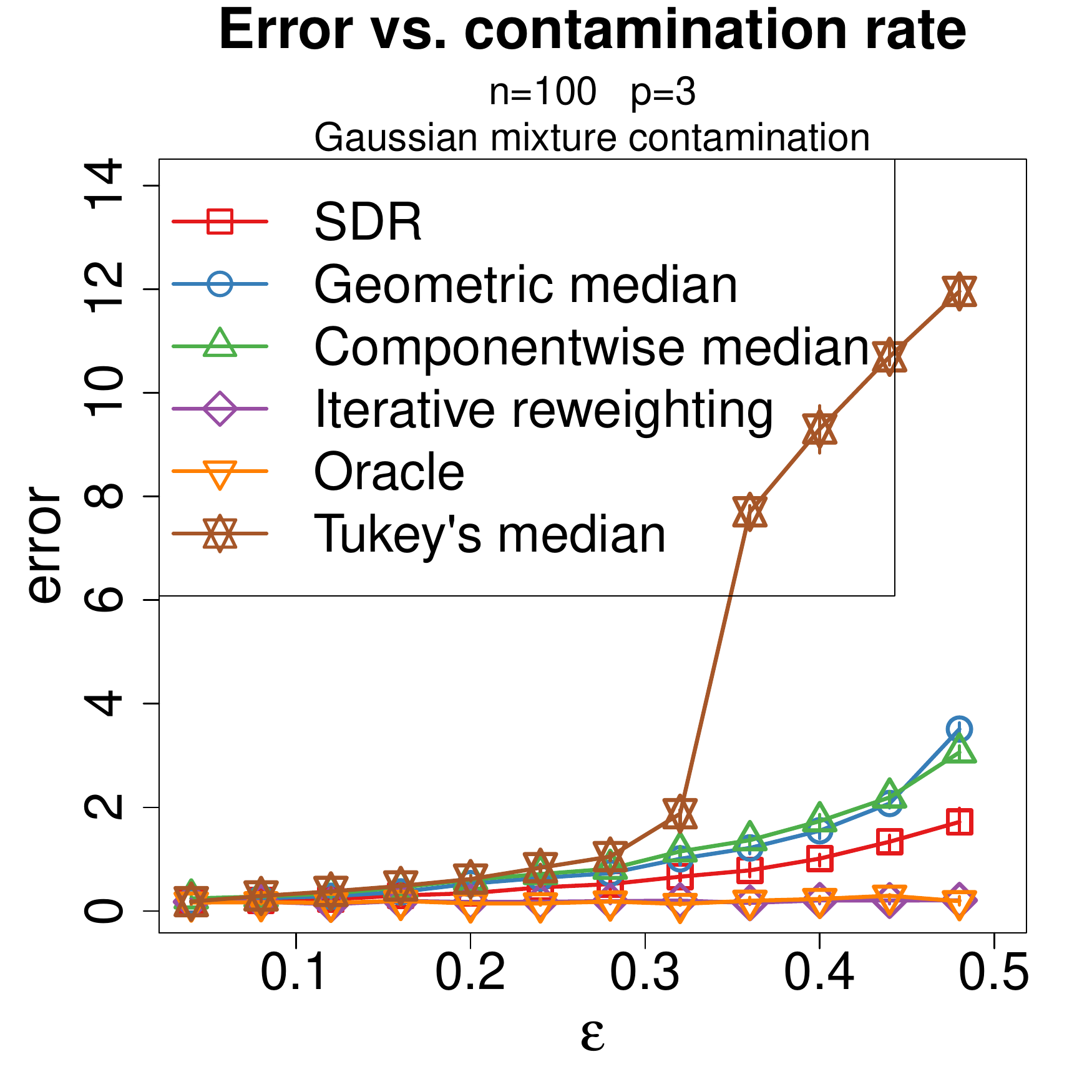}
    \caption{\label{fig:4} The left panel compares the robustness
    of various estimators by displaying the estimation error for
    different contamination rates under SCE scheme. SDR
    outperforms other estimators (even the oracle). Results
    displayed in the right panel are obtained by a similar
    experiment conducted for the GMC scheme. SDR is remarkably
    stable for different contamination schemes, while we see
    that SDR and Tukey's median may behave poorly for
    $\varepsilon>1/3$.}
\end{figure}




\section{Summary, related work and conclusion}
We have proved that the multivariate mean estimator
obtained by the iterative spectral dimension reduction method
enjoys several appealing properties in the setting
of sub-Gaussian observations subject to adversarial
contamination. More precisely, in addition to being rigid
transform equivariant and having breakdown point equal to $1/2$,
the estimator has been shown to achieve the nearly minimax rate.
Furthermore, the SDR estimator has low computational
complexity, confirmed by reported numerical experiments.
Indeed, its computational complexity is of the same order
as that of computing the sample covariance matrix and
performing a SVD on it. Presumably, at the cost of a moderate
drop in accuracy, further speed-ups can be obtained by
randomization \citep{Halko} in the spirit of the prior
work \citep{Yu_Cheng19,Depersin}.

Notably, we have proved that the SDR estimator achieves
the nearly optimal error rate without requiring the
precise knowledge of the contamination rate. It however
requires the knowledge of the covariance matrix. To
alleviate this constraint, we have also established
estimation guarantees in the case where an approximation
of the covariance matrix is used instead of the true one.
We have conducted numerical tests that show that the SDR
is both fast and accurate.

Many recent works studied the problem of robust estimation
in more complex high dimensional settings such
linear regression or sparse mean and covariance estimation;
see \citep{BalakrishnanDLS17,coldal, pensia2020robust,Thompson,
geoffrey2019erm,Goes,LiuSLC20,Cheng2021,Chinot20} and the
references therein. It is under current investigation whether
the results of the present paper can be extended to these
settings.

Another interesting avenue for future research is to find an
estimator that is rate-optimal, computationally tractable,
with breakdown point equal to $1/2$ and, in the same time,
asymptotically optimal in the sense that its risk is of
order $\sqrt{p/n}$ when $\varepsilon$ tends to zero. On
a related note, it would be interesting to push further the
exploration of second-order properties of the risk started
in \citep{minasyan2020}. Finally, an open question is how the
minimax risk blows-up when the contamination rate tends to
$\varepsilon$. For the SDR estimator studied in this work,
we established an upper bound of order $1/(1-\varepsilon)$.
However, we have no clue whether this is optimal, and our
intuition is that it is not.

\section{Proof of \Cref{thm:1}}
Before proving \Cref{thm:1}, we provide some
auxiliary lemmas and propositions.

\subsection{Bounding the projected error of the
average of filtered observations}

For any $J\subset [n]$, we define $\barz_J$ and
$\hat\bSigma{}^Z_J$ as the sample average and sample
covariance matrix of the subsample $\{\bZ_i: i\in J\}$,
that is
\begin{align}
    \barz_J = \frac1{|J|} \sum_{i\in J} \bZ_i,
    \qquad
    \hat\bSigma{}_J^Z = \frac1{|J|} \sum_{i\in J} \bZ_i
    \bZ_i^\top - \barz_J\barz_J^\top.
\end{align}
The main building block of the proof is the following
result.

\begin{proposition}\label{prop:projectedmean}
Let $\calS\subset [n]$ be an arbitrary set. We
define its subsets $\calS_{\calI} = \calS\cap\calI$ and
$\calS_{\calO} = \calS \cap\calO$. Let $\bZ_1,\ldots,\bZ_n$ and
$\bmu^Z$ be arbitrary points in $\mathbb R^q$ with
$q\ge 2$. Let $\bSigma^Z$ be an arbitrary $q\times q$
covariance matrix and let $\Proj_k$ be the projection
matrix projecting onto the subspace spanned by
the bottom $k$ eigenvectors of $\hat\bSigma{}_{\calS}^Z -
\bSigma^Z$, for $k=1,\ldots,q-1$. We have
\begin{align}
    \big\|\Proj_k(\barz_{\calS}-\bmu^Z)\big\|_2
    &\le \bigg\{{2\omega_\calO} \|\hat{\bSigma}{}^Z_{\calS_\calI}
    - \bSigma^Z\|_{\rm op} + \frac{\omega_\calO^2}{\omega_\calI}
    \bigg((\lambda_q-\lambda_1)(\bSigma^Z) +
    \frac{\delta_Z^2}{q-k}\bigg)\bigg\}^{1/2}\!\! + \big\|
    \Proj_k\barxi{}^Z_{\calS_{\calI} }\big\|_2,
\end{align}
where $\omega_{\calO} = |\calS_{\calO}|/|\calS|$,
$\omega_{\calI} = 1 - \omega_{\calO}$, $\bxi^Z_i = \bZ_i -\bmu^Z$
and $\delta_Z = \inf_{\bmu}\max_{i\in\calS} \|\bZ_i - \bmu^Z\|_2$.
Furthermore, if $|\calS| \le q-k$, then
\begin{align}
    \big\|\Proj_k(\barz_{\calS}-\bmu^Z)\big\|_2
    &\le \bigg\{{2\omega_\calO} \|\hat{\bSigma}{}^Z_{\calS_\calI}
    - \bSigma^Z\|_{\rm op} + \frac{\omega_\calO^2}{\omega_\calI}
    (\lambda_q-\lambda_1)(\bSigma^Z) \bigg\}^{1/2}\!\! + \big\|
    \Proj_k\barxi{}^Z_{\calS_{\calI} }\big\|_2.
\end{align}
\end{proposition}

\begin{proof}
Since there is no risk of confusion, we remove the
superscript $Z$ from $\bSigma{}^Z$, $\hat{\bSigma}{}^Z_{\calS}$
and so on. Since $\barz_{\calS}= \omega_{\calI}
\barz_{\calS_{\calI}} + w_{\calO}\barz_{\calS_{\calO}}$ yields
$\barz_{\calS} - \barz_{\calS_\calI} = w_{\calO}\big(
\barz_{\calS_{\calO}} - \barz_{\calS_{\calI}}\big)$, the
triangle inequality implies that
\begin{align}
    \big\|\Proj_k(\barz_{\calS}-\bmu)\big\|_2
    &\leq \big\|\Proj_k(\barz_{\calS} -
        \barz_{\calS_\calI}) \big\|_2 + \big\|
        \Proj_k(\barz_{\calS_{\calI}}-\bmu)\big\|_2\\
    &\leq w_{\calO} \big\| \Proj_k( \barz_{\calS_\calI} -
        \barz_{\calS_{\calO}}) \big\|_2
        + \big\|\Proj_k\barxi_{\calS_{\calI}}\big\|_2.
        \label{prop1:eq1}
\end{align}
Moreover, one can check that
\begin{align}
    \hat{\bSigma}_{\calS} & = \omega_{\calI}\hat{\bSigma}_{
    \calS_{\calI}} +\omega_{\calO}\hat{\bSigma}_{\calS_{\calO}}
    + \omega_{\calI}\omega_{\calO}(\barz_{\calS_{\calI}} -
    \barz_{\calS_{\calO}})^{\otimes2}   \label{eq1}\\
    &\succeq \omega_{\calI}\hat{\bSigma}_{\calS_{\calI}} +
    \omega_{\calI}\omega_{\calO}(\barz_{\calS_{\calI}} -
    \barz_{\calS_{\calO}})^{\otimes2}.
\end{align}
Hence, multiplying from left and right by $\Proj_k$ and computing
the largest eigenvalue of both sides, we get
\begin{align}
    \omega_\calI \omega_\calO \|\Proj_k
    (\barz_{\calS_{\calI}} - \barz_{\calS_\calO}) \|^2_2
    &\leq\lambda_q\big(\Proj_k ( \hat{\bSigma}_{\calS} -
    \omega_{\calI} \hat{\bSigma}_{\calS_\calI})
    \Proj_k^\top\big)\\
    &\leq \lambda_q\big( \Proj_k (\hat{\bSigma}_{\calS}-\bSigma)
    \Proj_k^\top \big) + \lambda_q\big(\bSigma - \omega_\calI
    \hat{\bSigma}_{\calS_\calI}\big)\\
    &\le \lambda_{k}(\hat{\bSigma}_{\calS}-\bSigma) + \omega_\calI
    \,\lambda_q\big(\bSigma - \hat{\bSigma}_{\calS_{\calI}}\big) +
    \omega_\calO\lambda_q(\bSigma).
    \label{eq2}
\end{align}
On the other hand, using the Weyl inequality (several times)
and the identity \eqref{eq1}, we get
\begin{align}
    \lambda_{k}(\hat{\bSigma}_{\calS}-\bSigma)
    &\leq \omega_\calI \lambda_q\big(
        \hat{\bSigma}_{\calS_{\calI}} -\bSigma\big) +
        \omega_\calO \lambda_{k}\big(\hat{\bSigma}_{
        \calS_\calO} - \bSigma + \omega_\calI
        (\barz_{\calS_{\calI}} - \barz_{\calS_\calO} )^{
        \otimes2}\big)\\
    &\leq
        \omega_\calI \lambda_q\big(
        \hat{\bSigma}_{\calS_{\calI}} -\bSigma\big) +
        \omega_\calO \lambda_{k+1}\big(\hat{\bSigma}_{
        \calS_\calO} - \bSigma\big) + \omega_\calI\lambda_{q-1}\big(
        (\barz_{\calS_{\calI}} - \barz_{\calS_\calO} )^{
        \otimes2}\big)\\
    &\leq \omega_\calI \lambda_q\big(\hat{\bSigma}_{\calS_{\calI}}
        - \bSigma\big) + \omega_\calO \lambda_{k+1} (\hat{\bSigma}_{\calS_\calO}) - \omega_\calO\lambda_1 (\bSigma).
        \label{eq3}
\end{align}
For the middle term of the right hand side, we can use the
following upper bound
\begin{align}
    \lambda_{k+1} (\hat{\bSigma}_{\calS_\calO})
        &\le \frac{\lambda_{k+1}(\hat{\bSigma}_{\calS_\calO})
            + \ldots+\lambda_q(\hat{\bSigma}_{\calS_\calO})}{q-k}\\
        &\le \frac{\tr(\hat{\bSigma}_{\calS_{\calO}})}{q-k} =
            \frac{1}{q-k} \tr\bigg(\frac{1}{|\calS_\calO|}
            \sum_{i\in \calS_{\calO}} (\bZ_i - \barz_{\calS_\calO}
            )^{\otimes 2}\bigg)\\
        &= \frac{1}{q-k}\inf_{\bmu}\tr\bigg(\frac1{|\calS_{\calO}|}
            \sum_{i\in\calS_{\calO}} (\bZ_i - \bmu)^{\otimes 2}
            \bigg)\\
        &= \inf_{\bmu}\frac1{(q-k)|\calS_{\calO}|}
        \sum_{i\in\calS_{\calO}} \|\bZ_i - \bmu\|_2^2
        \le  \frac{\delta_Z^2}{q-k}.
\end{align}
Combining \eqref{eq2}, \eqref{eq3} and the last display,
we get
\begin{align}
    \omega_\calI \omega_\calO \|\Proj_k(\barz_{\calS_\calI}
    - \barz_{\calS_\calO})\|_2^2
    &\leq 2 \omega_\calI\|\hat{\bSigma}_{\calS_{\calI}}
    - \bSigma\|_{\rm op} + \omega_\calO\big(\lambda_q(\bSigma)
    - \lambda_1(\bSigma)\big) + \frac{\omega_\calO \delta_Z^2}{q-k}.
\end{align}
Dividing both sides by $\omega_\calI\omega_\calO$,
we  arrive at
\begin{align}
    \|\Proj_k(\barz_{\calS_{\calI}}-\barz_{\calS_\calO})\|_2^2
    &\leq \frac{2}{
    \omega_\calO} \|\hat{\bSigma}_{\calS_{\calI}}
    - \bSigma\|_{\rm op} + \frac1{\omega_\calI} \bigg(
    \lambda_q(\bSigma) - \lambda_1(\bSigma) +
    \frac{\delta_Z^2}{q-k}\bigg).
\end{align}
Combining this inequality with \eqref{prop1:eq1}, we
obtain the first claim of the proposition. To get the second
claim, we simply remark that $\lambda_{k+1}(\hat\bSigma_{
\calS_\calO}) = 0$ since the rank of the matrix $\hat\bSigma_{
\calS_\calO}$ is less than $|\calS_\calO|\le |\calS| \le q-k$.
\end{proof}

\subsection{Bounding the error of the geometric
median of projected observations}

We assume in this section that $V$ is a linear subspace
of $\mathbb R^p$ of dimension $k$ and consider the
geometric median $\bmuGMV$ of the projected vectors
$\ProjV\bX_1,\ldots, \ProjV\bX_n$ that is
\begin{align}
    \bmuGMV\in\text{arg}\min_{\bmu\in\mathbb R^p}
    \sum_{i=1}^n \|\ProjV\bX_i-\bmu\|_2.
\end{align}
\begin{lemma}\label{lem:geommed}
    With probability at least $1-\delta$, for all
    linear subspaces $V\subset \mathbb R^p$, we have
    \begin{align}
        \frac{\|\bmuGMV - \ProjV\bmu^*\|_2}{\sqrt{\dim(V)}}
        &\leq \frac{2 \sqrt{\|\bSigma\|_{\rm op}}}{1-2\varepsilon}
        \bigg(1 + \frac{\sqrt{\rSigma} + \sqrt{2\log(1/\delta)}
        }{\sqrt{n}}\bigg).
    \end{align}
\end{lemma}
\begin{proof}
It follows from \cite[Lemma 2]{dalalyan2020allinone} that
\begin{align}
 \|\bmuGMV-\ProjV\bmu^*\|_2\leq \frac{2}{n(1-2\varepsilon)}
 \sum_{i=1}^n\|\ProjV\bxi_i\|_2.
\end{align}
Upper bounding the last term using Cauchy-Schwartz's inequality,
one obtains
\begin{align}
    \|\bmuGMV-\ProjV\bmu^*\|_2&\leq
    \frac{2}{\sqrt{n}(1-2\varepsilon)}\Big(\sum_{i=1}^n\|
    \ProjV\bxi_i\|_2^2\Big)^{1/2}.
\end{align}
Let $\be_1,\ldots,\be_k$ be any orthonormal basis
of $V$, with $k={\rm dim}(V)$. We have
\begin{align}
    \|\bmuGMV-\ProjV\bmu^*\|_2&\leq
    \frac{2}{\sqrt{n}(1-2\varepsilon)}\Big(\sum_{i=1}^n
    \sum_{\ell=1}^k|\be_\ell^\top\bxi_i|^2\Big)^{1/2}\\
    &\le \frac{2}{\sqrt{n}(1-2\varepsilon)}\Big(k
    \sup_{\|\be\|_2=1}\sum_{i=1}^n |\be^\top\bxi_i|^2\Big)^{1/2}\\
    &= \frac{2}{\sqrt{n}(1-2\varepsilon)}\Big(k
    \sup_{\|\be\|_2=1} \big\| [\bxi_1,\ldots,\bxi_n]^\top
    \be \big\|_2^2 \Big)^{1/2}\\
    &= \frac{2\sqrt{k}}{\sqrt{n}(1-2\varepsilon)}
     \big\|[\bxi_1,\ldots,\bxi_n]\big\|_{\rm op}.
\end{align}
By Corollary 5.35 in \citep{Vershynin2012IntroductionTT},
the inequality
\begin{align}
    \big\|[\bxi_1,\ldots,\bxi_n]\big\|_{\rm op} \leq
    \|\bSigma\|^{1/2}_{\rm op} \big(\sqrt{n} + \sqrt{\rSigma}
    + \sqrt{2\log(1/\delta)}\big)
\end{align}
holds with probability at least $1-\delta$. This yields
the desired inequality.
\end{proof}

\subsection{Bounding the number of filtered out observations}
\label{ssec:7.2}
Throughout this section, we assume without loss of
generality that $\|\bSigma\|_{\rm op}=1$.
\begin{lemma}\label{lem:filternoncontamine}
Let $\tau$ and $\delta$ be two numbers from $(0,1)$. Define
\begin{align}
    z = 1 + \frac{\sqrt{\rSigma} + \sqrt{2 \log(1/\delta)}
        }{\sqrt{n\tau}} + \sqrt{2 + 2\log\big(1/\tau\big)}.
\end{align}
With probability at least $1-\delta$, we have
\begin{align}
    \sup_{V} \sum_{i=1}^n \indic\big( \|\ProjV\bxi_i\|_2^2
    > z^2 \dim(V) \big) \le n\tau,
\end{align}
where the supremum is over all linear subspaces $V$ of
$\mathbb R^p$.
\end{lemma}
\begin{proof}
Let us define the random variable
\begin{align}
    T_n = \sup_{V} \sum_{i=1}^n \indic\big( \|\ProjV\bxi_i\|_2^2
    > z^2 \dim(V) \big).
\end{align}
In what follows, we write $d_V$ for the
dimension of the subspace $V$. We check that
\begin{align}
    \bfP(T_n > n\tau)
    &\leq\bfP\Big(\exists V\subset \mathbb R^p,
    \exists J\subset[n] \text{ with } |J|=n\tau, \text{ s.t. }
    \min_{i\in J}\|\ProjV\bxi_i\|_2^2\geq z^2\dim(V)\Big) \\
    &\leq\sum_{\substack{J\subseteq[n]\\|J|=n\tau}}
    \bfP\Big(\sup_{V}\min_{i\in J}\|\ProjV\bxi_i\|_2^2 /d_V \geq z^2\Big)\\
    &=\binom{n}{n\tau}\bfP\Big(\sup_{V}
    \min_{i\in \{1,\dots,n\tau\}}\|\ProjV\bxi_i\|_2^2 /d_V \geq z^2 \Big)\\
    &\leq\binom{n}{n\tau} \bfP\Big (\sup_{V}
    \frac{1}{n\tau d_V}
    \sum_{i=1}^{n\tau}\|\ProjV\bxi_i\|_2^2 \geq z^2
    \Big).\label{ineq:5}
\end{align}
Given a linear subspace $V\subset\mathbb R^p$ of
dimension $d_V$,  let $\be_1^V,\dots,\be_{d_V}^V$ be an
orthonormal basis of $V$. Using \eqref{ineq:5}, we get
\begin{align}
    \bfP(T_n\geq n\tau)
    &\leq\binom{n}{n\tau}\bfP\Big(\sup_{V}
    \frac{1}{n\tau d_V} \sum_{l=1}^{d_V} \sum_{i=1}^{n\tau}
    \big|\bxi_i^\top \be_l^V\big|^2\geq z^2\Big)\\
    &\leq \binom{n}{n\tau} \bfP\Big(\sup_{\|\be\|=1}
    \frac{1}{n\tau}\sum_{i=1}^{n\tau}\big|
    \bxi_i^\top \be\big|^2 \geq z^2\Big)
    \\
    &=\binom{n}{n\tau}\bfP\Big(\|[\bxi_1,\dots,\bxi_{n\tau}]
    \|_{\rm op}\geq z\sqrt{n\tau}\Big).
\end{align}
By \citep[Corollary 5.35]{Vershynin2012IntroductionTT}, the inequality
\begin{align}
\big\|[\bxi_1,\ldots,\bxi_{n\tau}]\big\|_{\rm op}
    \leq \big(\sqrt{n\tau} + \sqrt{\rSigma} +
    \sqrt{2\log(1/\delta_0)}\big)
\end{align}
holds with probability at least $1-\delta_0$. Taking
$\delta_0 = \delta/\binom{n}{n\tau}$ and using the inequality
$\log\binom{n}{n\tau}\leq n\tau\big(1 + \log(1/{\tau})\big)$,
we can conclude that $\bfP(T_n\geq n\tau)\leq\delta$.
This proves the lemma.
\end{proof}


\begin{lemma}\label{lem:filtercontamine}
    Let $\tau\in(0,1/2)$ and $\delta\in(0,1/2)$ be
    arbitrary. Set
    \begin{align}
        t &= \frac{3-2\varepsilon^*}{1-2\varepsilon^*}
            \bigg(1 + \frac{\sqrt{\rSigma} + \sqrt{2
            \log (2/\delta)}}{\sqrt{n\tau}}\bigg) +
            \sqrt{2 + 2\log\big(1/{\tau}\big)},
    \end{align}
    where $\varepsilon^*<1/2$. Assume that $\{\bX_i\}$ are
    drawn from GAC$(\bmu^*,\bSigma,\varepsilon)$ model with
    $\varepsilon\le \varepsilon^*$.
    Then, with probability at least $1-\delta$, for any
    linear subspace $V$ of $\mathbb R^p$, the inequalities
    \begin{align}
        &N_V =\sum_{i\in\calI} \indic\bigg( \frac{\|\ProjV\bX_i -
        \med_{V}\|_2}{\sqrt{\dim(V)}} \leq t\bigg)
        \geq n( 1-\varepsilon-\tau),\\
        &\frac{\|\ProjV\bmu^* - \med_{V} \|_2}{\sqrt{\dim(V)}}
        \le \frac{2}{1-2\varepsilon} \bigg(1 + \frac{
        \sqrt{\rSigma} + \sqrt{2\log(2/\delta)}}{\sqrt{n}}\bigg),
    \end{align}
    where $\med_{V}$ is the geometric median of $\ProjV\bX_1,
    \dots, \ProjV\bX_n$.
\end{lemma}
\begin{proof}  To avoid
unnecessary technicalities, we assume in this proof that
$n\tau$ is an integer. We also write
\begin{align}
    t_1 &= \frac{2}{1-2\varepsilon}\bigg(1 +
        \frac{\sqrt{\rSigma} + \sqrt{2 \log(2/\delta)}
        }{\sqrt{n}}\bigg)\\
    t_2 &= 1 + \frac{\sqrt{\rSigma} + \sqrt{2 \log( 2/\delta)}
    }{\sqrt{n\tau}} + \sqrt{2 + 2\log(1/{\tau})},
\end{align}
so that $t\ge  t_1 + t_2$. Simple algebra yields
\begin{align}
    N_V &= \sum_{i\in\calI} \indic \big( \| \ProjV
    \bxi_i+\ProjV\bmu^*-\med_{V}\|_2 \leq t\sqrt{\dim(V)}\big)\\
    &\geq \sum_{i=1}^n\indic \big( \| \ProjV
    \bxi_i+\ProjV\bmu^*-\med_{V}\|_2 \leq t\sqrt{\dim(V)}\big) -
    n\varepsilon\\
    &= n-n\varepsilon-\sum_{i=1}^n\indic \big( \| \ProjV
    \bxi_i+\ProjV\bmu^*-\med_{V}\|_2 > t\sqrt{\dim(V)}\big)\\
    &\geq n-n\varepsilon-\sum_{i=1}^n \indic \big( \| \ProjV
    \bxi_i\|_2 + \|\ProjV\bmu^*-\med_{V}\|_2 > t\sqrt{\dim(V)}\big),
\end{align}
where in the last step, we used the triangle inequality.
According to \Cref{lem:geommed}, on an event $\Omega_1$ of
probability at least $1-\delta/2$, we have $ \|\ProjV\bmu^* -
\med_{V} \|_2 \le t_1 \sqrt{\dim(V)}$ for every $V$.
This implies that on this event,
\begin{align}
    N_V &\geq n-n\varepsilon-\sum_{i=1}^n \indic \big( \| \ProjV
    \bxi_i\|_2  > t_2 \sqrt{\dim(V)}\big),\qquad \forall
    V\subset \mathbb R^p.
\end{align}
Using \Cref{lem:filternoncontamine}, we get that on an
event $\Omega_2$ of probability at least $1-\delta/2$,
the sum on the right  hand side of the last display is
less than $n\tau$. Therefore, on the intersection of the
events $\Omega_1$ and $\Omega_2$, we have $N_V\ge n(1-
\varepsilon - \tau)$ for every linear subspace $V$ of
$\mathbb R^p$.
\end{proof}

\subsection{Estimating the mean from a
low-dimensional adversarial projection}

In this section, we consider the following problem. We
assume that for a $q$ dimensional linear subspace $V$
of $\mathbb R^p$, which can depend on the sample $\bX_1,
\ldots,\bX_n$, we observe the projected data $\ProjV\bX_1
,\ldots,\ProjV\bX_n$. The goal is to estimate the
projected mean $\bmu^*_V =\ProjV\bmu^*\in\mathbb R^q$.
We  will estimate $\bmu^*_V$ by the mean of filtered
data points. More precisely, let $\bmuGMV$ be the
geometric median of $\ProjV\bX_1,\ldots,\ProjV\bX_n$.
We set
\begin{align}
    \calS_V = \{i\in[n] : \|\ProjV\bX_i -\bmuGMV\|_2
    \le t\sqrt{q}\},\quad N_V = |\calS_V|,
\end{align}
where $t$ is a real number. We estimate $\bmu^*_{V}$ by
\begin{align}
    \hat\bmu_{V} = \ProjV\barx_{\calS_V} =
    \frac1{N_V} \sum_{i\in\calS_V}
    \ProjV\bX_i.
\end{align}

\begin{lemma}\label{lem:lowdim}
For every positive threshold $t>0$, we have
\begin{align}
    \|\hat\bmu_{V} - \bmu^*_V \|_2
    &\le \frac{\big\|\ProjV\bar\bxi\big\|_2}{N_V}
    + \frac1{N_V} \bigg\|\sum_{i\in\mathcal
    S_V^c\cup \calO}\ProjV\bxi_i\bigg\|_2 + \frac{n
    \varepsilon (t\sqrt{q} + \|\bmuGMV-\bmu^*_{V}\|_2)}{
    N_V}.
\end{align}
\end{lemma}
\begin{proof}
For this estimator, using the triangle inequality and the
fact that $\bX_i = \bmu^* + \bxi_i$ for every
$i\in\calI$, we have
\begin{align}
    \|\hat\bmu_{V} - \bmu^*_V \|_2 &=
    \frac1{N_V} \bigg\|\sum_{i\in\calS_V}
    (\ProjV\bX_i-\ProjV\bmu^*)\bigg\|_2\\
    &\le \frac1{N_V} \bigg\|\sum_{i\in\calS_V
    \cap \calI}\ProjV\bxi_i\bigg\|_2 +
    \frac1{N_V} \bigg\|\sum_{i\in\calS_V
    \cap \calO}\ProjV(\bX_i - \bmu^*)\bigg\|_2\\
    &\le \frac1{N_V} \bigg\|\sum_{i\in\calS_V
    \cap \calI} \ProjV\bxi_i\bigg\|_2 + \frac{n\varepsilon
    (t\sqrt{q}+ \|\bmuGMV-\bmu^*_{V}\|_2)}{N_V} .
\end{align}
Using once again the triangle inequality, we arrive at
\begin{align}
    \|\hat\bmu_{V} - \bmu^*_V \|_2
    &\le \frac{\big\|\ProjV\bar\bxi\big\|_2}{N_V}
    + \frac1{N_V} \bigg\|\sum_{i\in\mathcal
    S_V^c\cup \calO}\ProjV\bxi_i\bigg\|_2 + \frac{n
    \varepsilon (t\sqrt{q} + \|\bmuGMV-\bmu^*_{V}\|_2)}{
    N_V} \\
    &\le \frac{\big\|\ProjV\bar\bxi\big\|_2}{N_V}
    + \frac1{N_V} \bigg\|\sum_{i\in\mathcal
    S_V^c\cup \calO}\ProjV\bxi_i\bigg\|_2 + \frac{n
    \varepsilon (t\sqrt{q} + \|\bmuGMV-\bmu^*_{V}\|_2)}{
    N_V}.
\end{align}
This completes the proof.
\end{proof}

\begin{lemma}
For every positive threshold $t>0$, we have
\begin{align}
    \|\hat\bmu_{V} - \bmu^*_V \|_2
    &\le \max_{|\mathcal S|\ge|\mathcal
    S_V\cap \calI|} \bigg\|\frac1{|\mathcal S|}\sum_{i\in\mathcal
    S}\ProjV\bxi_i\bigg\|_2 + \frac{n
    \varepsilon (t\sqrt{q} + \|\bmuGMV-\bmu^*_{V}\|_2)}{
    N_V}.
\end{align}
\end{lemma}
\subsection{Bounding stochastic errors}

Throughout this section, without loss of generality, we
assume that $\|\bSigma\|_{\rm op} = 1$.

\begin{lemma}\label{lem:mean}
For any positive integer $m\le n$ and any $t>0$,
we have
\begin{align}
    \bfP\bigg(\max_{|S|\ge n-m}\Big\|\frac1{|S|}\sum_{i\in S}\Proj\,\bxi_i\Big\|_2 \leq \frac{n\|\Proj\bar\bxi_n
    \|_2}{n-m}  +  \frac{\sqrt{m(2\rSigma+3t)} +
    m\sqrt{3\log(2ne/m)}}{n-m}\bigg)
    & \geq 1 - e^{-t}.
\end{align}
\end{lemma}
\begin{proof}
Using the triangle inequality, one has
\begin{align}
    \frac{1}{|S|}\Big\| \sum_{i\in S}\Proj\bxi_i\Big\|_2
    &\leq\frac{1}{n-m} \Big\| \sum_{i\in S}\Proj\bxi_i\Big\|_2
    =\frac{1}{n-m}\Big\|\sum_{i=1}^n\Proj\bxi_i-
    \sum_{i\in S^C}\Proj\bxi_i\Big\|_2\\
    &\leq\frac{1}{n-m}\Big\|\sum_{i=1}^n\Proj\bxi_i\Big\|_2
    + \frac{1}{n-m}\Big\|\sum_{i\in S^C}\Proj\bxi_i\Big\|_2\\
    &\leq\frac{n\|\Proj\bar\bxi_n\|_2}{n - m} +
    \frac{1}{n-m}\max_{|J|\le m} \Big\| \sum_{i\in J}
    \bxi_i \Big\|_2.\label{eq:8}
\end{align}
For $s\in[1,m]$, we choose $t_s$ by
\begin{align}
    t_s = 3s\log\Big(\frac{2ne}{s}\Big) + 3t,
\end{align}
so that $t_s\le t_m = 3m\log\big(\frac{2ne}{m}\big) + 3t$ and
\begin{align}
    \Big(\frac{ne}{s}\Big)^s e^{-t_s/3} \leq 2^{-s}e^{-t}.
\end{align}
For every $J$ of cardinality $m$, the random variable $\|\sum_{j\in J}\bxi_j\|^2$ has the same distribution as $m\sum_{j=1}^p\lambda_j(\bSigma)\alpha_j^2$, where $\alpha_1,\dots,\alpha_p$ are i.i.d. standard Gaussian.
Hence, using the union bound, the well-known upper bound
on the binomial coefficients and \cite[Lemma 8]{Comminges}, we have
\begin{align}
    \bfP\Big(\max_{|J|\leq m} \Big\|\sum_{i \in J}
    \bxi_i \Big\|_2^2 \geq m(2\rSigma + t_m)\Big)
    &\leq \sum_{s=1}^{m} \binom{n}{s} \bfP\Big( \Big\|
    \sum_{i=1}^s \bxi_i\Big\|_2^2 \geq m(2\rSigma+t_m)\Big)\\
    &\leq \sum_{s=1}^{m} \Big(\frac{ne}{s}\Big)^s \bfP\Big(
    \Big\|\sum_{i=1}^s \bxi_i\Big\|_2^2 \geq s(2\rSigma + t_s)
    \Big)\\
    &\leq \sum_{s=1}^{m} \Big(\frac{ne}{s}\Big)^s e^{-t_s/3}
    \leq e^{-t}.
\end{align}
This entails that with probability at least $1-e^{-t}$, we have
\begin{align}
    \max_{|J|\leq m}  \Big\|\sum_{i \in J} \bxi_i \Big\|_2
        &\leq \sqrt{m(2\rSigma + t_m)}
        \leq \sqrt{m(2\rSigma + 3t)} + m\sqrt{3\log(2ne/m)}.
\end{align}
Combining this inequality with \eqref{eq:8}, we get the
claim of the lemma.
\end{proof}

In the two next lemmas, given a set $\calS\subset [n]$,
we look at the sample average and sample covariance matrix
of the subsample $\{\bX_i: i\in \calS\}$,
\begin{align}
    \barx_\calS = \frac1{|\calS|} \sum_{i\in\calS} \bX_i,\qquad
    \hat\bSigma{}_\calS = \frac1{|\calS|} \sum_{i\in \calS} \bX_i
    \bX_i^\top - \barx_\calS\barx_\calS^\top.
\end{align}

\begin{lemma}\label{lem:sigmaTr2}
There exists a positive constant $A$ such that, for
any positive integer $m\le n$ and any $t\geq1$, with
probability at least $1-2e^{-t}$, the inequality
\begin{align}
    \|\hat{\bSigma}_{\calS} - \bSigma\|_{\rm op}\leq
    A\frac{\sqrt{n\rSigma} + \rSigma + m\log(2ne/m)+2t}{n-m} + \big\|\barxi_\calS\big\|_2^2
\end{align}
is satisfied for every $\calS\subset [n]$ of cardinality
$\geq n-m$.
\end{lemma}
\begin{proof}
The triangle inequality implies
\begin{align}
    \|\hat{\bSigma}_{\calS} - \bSigma\|_{\rm op}&\leq \Big\|\frac{1}{|\calS|}\sum_{i\in\calS}(\bX_{i}-\bmu)^{\otimes 2}-\bSigma\Big\|_{\rm op}+\big\|(\bmu-\barx_\calS)^{\otimes 2}\big\|_{\rm op}\\
    &=\Big\|\frac{1}{|\calS|}\sum_{i\in\calS}\bxi_i\bxi_i^\top-\bSigma\Big\|_{\rm op} + \big\|\barxi_\calS\big\|_2^2.
    \label{triangle}
\end{align}
Using again the triangle inequality, one gets
\begin{align}
    \Big\|\sum_{i\in\calS}\bxi_i\bxi_i^\top-\bSigma\Big\|_{\rm op}&\leq\Big\|\sum_{i\in\calS}\big(\bxi_i\bxi_i^\top-\bSigma\big)\Big\|_{\rm op}\\
    &\leq\Big\|\sum_{i=1}^n\big(\bxi_i\bxi_i^\top-\bSigma\big)\Big\|_{\rm op}+\Big\|\sum_{i\in\calS^C}\big(\bxi_i\bxi_i^\top-\bSigma\big)\Big\|_{\rm op}\\
    &\leq\Big\|\sum_{i=1}^n\big(\bxi_i\bxi_i^\top-\bSigma\big)\Big\|_{\rm op} + \max_{|J|\leq m}\Big\|\sum_{i\in J}\big(\bxi_i\bxi_i^\top-\bSigma\big)\Big\|_{\rm op}.
    \label{sigma}
\end{align}
In view of \cite[Theorems 4 and 5]{KoltLoun},
one can show that there exists a positive universal
constant $A_1$ such that for every $t\geq1$ and
every set $J$ of cardinality $s$, the inequality
\begin{align}
    \Big\|\sum_{j\in J}\bxi_j\bxi_j^\top-s\bSigma
    \Big\|_{\rm op}\leq A_1(\sqrt{m\rSigma}+\rSigma + t)
\end{align}
is satisfied with probability at least $1-e^{-t}$.
We define $t_s$ by
\begin{align}
    t_s = s\log\Big(\frac{2ne}{s}\Big) + t,
\end{align}
so that $t_s\le t_m = m\log\big(\frac{2ne}{m}\big) + t$ and
\begin{align}
    \Big(\frac{ne}{s}\Big)^s e^{-t_s} \leq 2^{-s}e^{-t}.
\end{align}
Applying the union bound and the well-known upper bound
on the binomial coefficients, this yields
\begin{align}
    \bfP\Big(\max_{|J|\leq m}\Big\|\sum_{j\in J}\big(\bxi_j\bxi_j^\top&-\bSigma\big)\Big\|_{\rm op}\geq A_1(\sqrt{m\rSigma}+\rSigma + t_m) \Big)\\
    &\leq\sum_{s=1}^m\binom{n}{s}\bfP\Big(\Big\|\sum_{j=1}^s\bxi_j\bxi_j^\top-s\bSigma\Big\|_{\rm op}\geq A_1(\sqrt{m\rSigma}+\rSigma + t_s)\Big)\\
    &\leq\sum_{s=1}^m\Big(\frac{ne}{s}\Big)^s e^{-t_s}\leq e^{-t}.
\end{align}
One deduces from \eqref{sigma} that, with probability at least $1-2e^{-t}$,
\begin{align}
    \Big\|\frac{1}{|\calS|}\sum_{i\in\calS}\bxi_i\bxi_i^\top
    -\bSigma\Big\|_{\rm op}
    &\leq A_1\frac{(\sqrt{n}+\sqrt{m})\sqrt{\rSigma} + 2\rSigma + m\log(2ne/m)+2t}{n-m}\\
    &\leq A\frac{\sqrt{n\rSigma} + \rSigma + m\log(2ne/m)+2t}{n-m}.
\end{align}
Combining this with \eqref{triangle}, one gets the claim of the lemma.
\end{proof}

\begin{lemma}\label{lem:sigma_p}
For any positive integer $m\le n$ and any $t>0$, with probability
at least $1-4e^{-t}$, the inequality
\begin{align}
    \|\hat{\bSigma}_{\calS} - \bSigma\|_{\rm op}&\leq \frac{5p+\big(8\log(2ne/m)+2\big)m + 7t}{n-m} + 2\frac{\sqrt{p}+\sqrt{t}}{\sqrt{n}-\sqrt{m}}+ \big\|\barxi_\calS\big\|_2^2
\end{align}
is satisfied for every $\calS\subset [n]$ of cardinality
$\geq n-m$.
\end{lemma}

\begin{proof}
    In this proof, without loss of generality, we assume that the matrix $\bSigma$ is invertible.
    In view of \eqref{triangle} and \eqref{sigma}, we have
    \begin{align}\label{eq:sigmaop}
        \|\hat{\bSigma}_{\calS} - \bSigma\|_{\rm op}\leq \Big\|\frac{1}{|\calS|}\sum_{i=1}^n\big(\bxi_i\bxi_i^\top
        -\bSigma\big)\Big\|_{\rm op} + \max_{|J|\leq m} \Big\|
        \frac{1}{|\calS|} \sum_{i\in J}\big(\bxi_i\bxi_i^\top -
        \bSigma\big)\Big\|_{\rm op}+ \big\|\barxi_\calS\big\|_2^2.
    \end{align}
    Let us define $\bzeta_i:=\bSigma^{-1/2}\bxi_i$ for all $i\in[n]$.
    For every set $J$ of cardinality $s$, it holds that
    \begin{align}
        \Big\|\sum_{i\in J}\big(\bxi_i\bxi_i^\top-\bSigma\big)\Big\|_{\rm op}&\leq\big\|\bSigma\big\|_{\rm op}\Big\|\sum_{i\in J}\big(\bzeta_i\bzeta_i^\top-\bfI_p\big)\Big\|_{\rm op}\\
        &=\max\Big(\lambda_{\max}\big(\sum_{i\in J}\bzeta_i\bzeta_i^\top\big)-s,s-\lambda_{\min}\big(\sum_{i\in J}\bzeta_i\bzeta_i^\top\big)\Big)\\
        &=\max\Big(\sigma^2_{\max}\big(\bzeta_{J}\big)-s,
        s-\sigma^2_{\min}\big(\bzeta_{J}\big)\Big),
    \end{align}
    where $\bzeta_{J}$ is the $s\times n$ random matrix
    obtained by concatenating the vectors $\bzeta_i$ with
    $i\in J$. By \cite[Corollary 5.35]{Vershynin2012IntroductionTT},
    we know that for every $x>0$
    \begin{align}
        \sqrt{s}-\sqrt{p}-x\leq \sigma_{\min}(\bzeta_{J})\leq \sigma_{\max}(\bzeta_{J})\leq\sqrt{s}+\sqrt{p}+x
    \end{align}
    is true with probability at least $1-2e^{-x^2/2}$.
    This yields\footnote{We provide the argument only in the case $\sqrt{s}\ge \sqrt{p} + x$, but the conclusion is true for every value $s$.}
    \begin{align}
        \Big\|\sum_{i\in J}\big(\bxi_i\bxi_i^\top-\bSigma\big)\Big\|_{\rm op}&\leq
        \max\big((\sqrt{p}+x)(2\sqrt{s}+\sqrt{p}+x),(\sqrt{p}+x)(2\sqrt{s}-\sqrt{p}-x)\big)\\
        &\leq p+x^2+2\sqrt{ps}+2x\sqrt{p}+2x\sqrt{s}
    \end{align}
    with probability at least $1-2e^{-x^2/2}$. By applying the same
    technique as in the proof of \Cref{lem:sigmaTr2}, we can set
    \begin{align}
    t_s = 2\sqrt{s\log\Big(\frac{2ne}{s}\Big) + t},
    \end{align}
    and obtain
    \begin{align}
    \bfP\Big(\max_{|J|\leq m}\Big\|\sum_{j\in J}\big(\bxi_j\bxi_j^\top&-\bSigma\big)\Big\|_{\rm op}\geq p+t_m^2+2\sqrt{pm}+2t_m\sqrt{p}+2t_m\sqrt{m} \Big)\leq 2e^{-t}.
\end{align}
    Hence, going back to \eqref{eq:sigmaop}, we can show that the inequalities
    \begin{align}
        \|\hat{\bSigma}_{\calS} - \bSigma\|_{\rm op}&\leq\frac{p+t+2\sqrt{pn}+2\sqrt{tp}+2\sqrt{tn}}{n-m}+\frac{p+4t+4m\log(2ne/m)+2\sqrt{pm}}{n-m}
        \\
        &\quad+\frac{4(\sqrt{p}+\sqrt{m})\sqrt{m\log(2ne/m)+t}}{n-m}+\big\|\barxi_\calS\big\|_2^2\\
        &\leq \frac{5p+8m\log(2ne/m) + 2m + 7t}{n-m} + \frac{2(\sqrt{p} +
        \sqrt{t})(\sqrt{n}+\sqrt{m})}{n-m}+\big\|\barxi_\calS\big\|_2^2
    \end{align}
    hold with probability at least $1-4e^{-t}$, and this proves the lemma.
\end{proof}

\subsection{Putting all the pieces together}
All the ingredients provided, we can now compile the complete proof of \Cref{thm:1}.

Taking $\bfU_L:=\bfV_L$, the algorithm detailed in \eqref{alg:3} returns $\bmuSDR=\sum_{\ell=0}^L\hat\bmu^{(\ell)}$ with $\hat\bmu^{(\ell)}\in\mathscr U_\ell={\rm Im}(\bfV_\ell\bfU_\ell^\top)$ for every $\ell\in\{0,\ldots,L\}$ where the two-by-two orthogonal subspaces $\mathscr U_0,\ldots,\mathscr U_L$ span the whole space $\mathbb{R}^p$.
This allows us to decompose the risk:
\begin{align}
    \big\|\bmuSDR - \bmu^*\big\|_2^2 &= \sum_{\ell=0}^L
    \big\|\hat\bmu^{(\ell)} - \Proj_{\mathscr U_\ell}
    \bmu^* \big\|_2^2 = \sum_{\ell=0}^L \big\|
    \Proj_{\mathscr U_\ell}(\barx_{\ell} - \bmu^*)
    \big\|_2^2\\
    &=\sum_{\ell=0}^L\big\| \Proj_\ell (\bfV_{\ell}^\top
    \barx_{\ell}  - \bfV_{\ell}^\top\bmu^*)\big\|_2^2,
\end{align}
where $\Proj_\ell:=\bfU_{\ell}^\top\bfU_\ell$ is the projection matrix projecting onto the subspace of $\mathbb{R}^{p_{\ell}}$ spanned by the bottom $p_\ell - {p_{\ell+1}}$ eigenvectors of $\bfV_\ell^\top(\hat\bSigma{}^{(\ell)} - \bSigma)\bfV_\ell$ for $\ell=0,\ldots,L$ with the convention that $p_{L+1}=0$.

For $\ell\in\{0,\ldots,L-1\}$, we intend to apply
\Cref{prop:projectedmean} to $\bZ_i = \bfV_{\ell}^\top\bX_i$
and $\bmu^Z = \bfV_{\ell}^\top\bmu^*$ in order to upper bound
the error term $\Err_\ell := \|\Proj_\ell(\bfV_{\ell}^\top
\barx_{\ell}  - \bfV_{\ell}^\top \bmu^* )\big\|_2$. Using
the inequalities
\begin{align}
    \|\bfV^\top
    (\hat\bSigma{}^{(\ell)} - \bSigma)\bfV\|_{\rm op} \le
    \|\hat\bSigma{}^{(\ell)} - \bSigma\|_{\rm op},
    \quad
    \lambda_{p_\ell} (\bfV^\top \bSigma \bfV)
    \le \lambda_{p} (\bSigma),
    \quad
    \lambda_{1} (\bfV^\top \bSigma \bfV)
    \ge \lambda_{1} (\bSigma)
\end{align}
that are true for every orthogonal matrix $\bV$,
and keeping in mind the definition of $\Proj_\ell$,
we get
\begin{align}
    \Err_\ell &\le
    \bigg\{2\omega_\calO \|
    \hat\bSigma{}^{(\ell)} - \bSigma\|_{\rm op}
    + \frac{\omega_\calO^2}{1-\omega_\calO}
    \bigg((\lambda_{p}-\lambda_1)(\bSigma) + \frac{\delta_\ell^2}{p_{\ell+1}}\bigg)
    \bigg\}^{1/2} + \big\|   \Proj_\ell\bfV_\ell^\top
    \barxi_{\calS^{(\ell)}_\calI }\big\|_2,
\end{align}
where we have used the notation
\begin{align}
    \omega_{\calO} = \max_{\ell} \frac{|\calS^{(\ell)}\cap\calO|}
    {|\calS^{(\ell)}|},\qquad
    \bxi_i = \bX_i -\bmu^*
\end{align}
and
$\delta_\ell = \inf_{\bmu}
\max_{i\in\calS^{(\ell)}} \|\bfV_\ell^\top (\bX_i -
\bmu)\|_2$. Note that when $\calO$ and $(\calS^{(\ell)
}\cap \calI)^c$ are of cardinality less than $n\varepsilon$
and $n(\varepsilon + \tau)$, respectively, we have
\begin{align}
    \frac{|\calS^{(\ell)}|}{|\calS^{(\ell)}\cap \calO|}
    = \frac{|\calS^{(\ell)}\cap \calI| + |\calS^{(\ell)}
    \cap \calO| }{|\calS^{(\ell)}\cap \calO|} =
    \frac{|\calS^{(\ell)}\cap \calI|}{|\calS^{(\ell)}\cap \calO|} + 1
    \ge \frac{n(1-\varepsilon-\tau)}{n\varepsilon} +1
    = \frac{1-\tau}{\varepsilon}
\end{align}
and, therefore, $\omega_{\calO}\le \varepsilon/(1-\tau)$. We set
$\eta := \varepsilon + \tau\le 3/4$ and apply
\Cref{lem:filtercontamine} to infer that $\omega_{\calO} \leq
\omega_{\calO}/(1-\omega_{\calO}) \le \varepsilon/(1-\eta)\le
 4\varepsilon$ is true with probability
at least $1-\delta$. Furthermore, we know that $\delta_\ell\leq\max_{i\in\calS^{(\ell)}}\|\bfV_\ell^\top\bX_i
- \bar\bmu^{(\ell)}\|_2\le t\sqrt{p_\ell}$. This yields
\begin{align}
    \Err_\ell &\leq \Big\{8\varepsilon
    \|\hat\bSigma{}^{(\ell)} - \bSigma\|_{\rm op}
    + 16\varepsilon^2 \Big((\lambda_{p}-\lambda_1) (\bSigma)
    + \frac{t^2p_{\ell}}{p_{\ell+1}}\Big)\Big\}^{1/2}
    + \big\|\Proj_{\mathscr U_\ell}\barxi_{\calS^{(\ell)
    }_\calI}\big\|_2.
\end{align}
Let us introduce the shorthand
\begin{align}
    T_1 =  \max_{\ell\in[L]}
    \|\hat\bSigma{}^{(\ell)} - \bSigma\|_{\rm op}
    + \varepsilon (\lambda_{p}-\lambda_1) (\bSigma).
\end{align}
This leads to
\begin{align}\label{Err:l}
    \Err_\ell &\leq \Big\{8\varepsilon T_1 +
    \frac{16\varepsilon^2 t^2p_{\ell}}{p_{\ell+1}}
    \Big\}^{1/2} + \| \Proj_{\mathscr U_\ell}
    \barxi_{\calS^{(\ell) }_\calI}\big\|_2.
\end{align}
For the last error term, since $p_L=1$ we have by \Cref{lem:lowdim}
\begin{align}
  \Err_L&\le \big\|\Proj_{\mathscr U_L} \bar\bxi_{\calS_\calI^{(L)}} \big\|_2 + \frac{n
    \varepsilon (t\sqrt{p_L} + \big\|\Proj_{\mathscr U_L}\bmu^* - \med_{\mathscr U_L}\big\|_2)}{
    |\mathcal S^{(L)}|}\\
    &\le \big\|\Proj_{\mathscr U_L} \bar\bxi_{\calS_\calI^{(L)}} \big\|_2 + \frac{
    \varepsilon t + \varepsilon\big\|\Proj_{\mathscr U_L}\bmu^* - \med_{\mathscr U_L}\big\|_2}{1-\eta}\\
    &\le \big\|\Proj_{\mathscr U_L} \bar\bxi_{\calS_\calI^{(L)}} \big\|_2 +
    4\varepsilon t + 4\varepsilon\big\|\Proj_{\mathscr U_L}\bmu^* -\med_{\mathscr U_L}\big\|_2.\label{Err:L}
\end{align}
Combining \eqref{Err:l}, \eqref{Err:L}, inequality
$p_\ell\le e p_{\ell+1}$, as well as
the Minkowski inequality, we get
\begin{align}
    \big\|\bmu^* &- \bmuSDR\big\|_2  = \bigg\{\sum_{\ell=0}^L
    \Err_\ell^2\bigg\}^{1/2}\\
    &\le \bigg\{ 8\varepsilon L ( T_1 + e \varepsilon
    t^2) + 16\varepsilon^2\big(t + \big\|\Proj_{\mathscr
    U_L}\bmu^* -\med_{\mathscr U_L}\big\|_2\big)^2\bigg\}^{1/2}
    + \bigg\{\sum_{\ell=0}^L \| \Proj_{\mathscr U_\ell}
    \barxi_{\calS^{(\ell) }_\calI}\big\|_2^2\bigg\}^{1/2}\\
    &\le 2\sqrt{2\varepsilon L T_1} + 9\varepsilon t
    \sqrt{L} + 4\varepsilon \big\|\Proj_{\mathscr
    U_L}\bmu^* -\med_{\mathscr U_L}\big\|_2 +
    \bigg\{\sum_{\ell=0}^L \| \Proj_{\mathscr U_\ell}
    \barxi_{\calS^{(\ell) }_\calI}\big\|_2^2\bigg\}^{1/2}\!\!
    \!\!.
    \label{eq:14}
\end{align}
To ease notation, 
let us set
\begin{align}
    {\sf r}_n = \Big(\frac{2\rSigma+3\log(2/\delta)}{n}\Big)^{1/2}.
\end{align}
In view of \Cref{lem:mean}, with
probability at least $1-\delta$, we have
\begin{align}
    \bigg\{\sum_{\ell=0}^L \| \Proj_{\mathscr U_\ell}
    \barxi_{\calS^{(\ell) }_\calI}\big\|_2^2\bigg\}^{1/2}
    &\le
    \bigg\{\sum_{\ell=0}^L \bigg( \frac{\|\Proj_{\mathscr U_\ell}\bar\bxi_n
    \|_2}{1 - \eta}  +  \frac{{\sf r}_n\sqrt{\eta} + \eta\sqrt{3 \log(2e/\eta)}}{1-\eta}\bigg)^2\bigg\}^{1/2}\\
    &\le \bigg\{\sum_{\ell=0}^L \big( 4\|\Proj_{\mathscr U_\ell}
    \bar\bxi_n \|_2  +  4{\sf r}_n\sqrt{\eta} + 10\eta\sqrt{
    \log(2/\eta)}\big)^2\bigg\}^{1/2}\\
    &\le 4\|\bar\bxi_n \|_2 + 4{\sf r}_n\sqrt{\eta L} +
    10\eta\sqrt{L\log(2/\eta)}.
\end{align}
Since the random variable $\|\barxi_n\|_2^2$ has the same distribution as $\frac1n\sum_{j=1}^p\lambda_j(\bSigma)\gamma_j^2$, where $\gamma_1,\dots,\gamma_p$ are i.i.d. standard Gaussian, by \cite[Lemma 8]{Comminges} we have
\begin{align}
    \|\barxi_n\|_2^2\le\frac{2\rSigma+3\log(2/\delta)}{n}
    ={\sf r}_n^2
\end{align}
with probability at least $1-\delta$. Therefore, with
probability at least $1-2\delta$,
\begin{align}
    \bigg(\sum_{\ell=0}^L \| \Proj_{\mathscr U_\ell}
    \barxi_{\calS^{(\ell) }_\calI}\big\|_2^2\bigg)^{1/2}
    &\le  4{\sf r}_n \big(1 + \sqrt{L\eta}\big) + 10\eta
    \sqrt{L\log(2/\eta)}.\label{eq:14bis}
\end{align}
Next, \Cref{lem:filtercontamine} and the fact that
$p_L = \dim(\mathscr U_L) = 1$  imply that with probability at least $1-\delta$
\begin{align}
    \big\|\Proj_{\mathscr
    U_L}\bmu^* -\med_{\mathscr U_L}\big\|_2 \le
    \frac{2(1+ \sqrt2{\sf r}_n)}{1-2\varepsilon}.\label{eq:15}
\end{align}
Recall that we have chosen $t$ in such a way that
\begin{align}
    t \le \frac{3(1 + \sqrt2{\sf r}_n/\sqrt{\tau})}{1-2\varepsilon^*}
        +  1.6\sqrt{\log(2/{{\tau}})}.\label{eq:16}
\end{align}
Combining \eqref{eq:14}, \eqref{eq:14bis}, \eqref{eq:15} and \eqref{eq:16},
we arrive at the inequality
\begin{align}
    \big\|\bmuSDR - \bmu^*\big\|_2
    &\le 2\sqrt{2\varepsilon L T_1} + 9\varepsilon t
    \sqrt{L} +  \frac{8\varepsilon(1+\sqrt{2}{\sf r}_n)}{1-2\varepsilon}
    + 4{\sf r}_n \big(1 + \sqrt{L\eta}\big) + 10\eta
    \sqrt{L\log(2/\eta)}\\
    &\le 2\sqrt{2\varepsilon L T_1} +
    \frac{27\varepsilon \sqrt{L}(1 + \sqrt2{\sf r}_n/\sqrt{\tau})}{1 -
    2\varepsilon^*} + 14.4\varepsilon \sqrt{L\log(2/\tau)} \\
    &\quad +  \frac{8\varepsilon(1+\sqrt{2}{\sf r}_n)}{1-2\varepsilon}
    + 4{\sf r}_n \big(1 + \sqrt{L\eta}\big) + 10\eta
    \sqrt{L\log(2/\eta)}
\end{align}
that holds with probability at least $1-3\delta$. In the
upper bound obtained above, only the term $T_1$ remains
random. We can upper bound this term using
\Cref{lem:sigmaTr2}. It implies that with probability
at least $1-2\delta$, we have
\begin{align}
    T_1 &\le  A\frac{\sqrt{n\rSigma} + \rSigma + n\eta
    \log(2e/\eta) + 2\log(1/\delta)}{n(1-\eta)} + \big(
    4{\sf r}_n(1+\sqrt{\eta}) + 10\eta\sqrt{\log(2/\eta)}
    \big)^2 + \varepsilon\\
    &\le 2A\big(\sqrt{2}{\sf r}_n + {\sf r}_n^2  + 4\eta
    \log(2/\eta)\big) + \big (7.5{\sf r}_n + 10\eta
    \sqrt{\log(2/\eta)}\big)^2 + \varepsilon.
\end{align}
Consequently,
\begin{align}
    \sqrt{\varepsilon T_1} &\le  \big\{2A\varepsilon
    \big(\sqrt{2}{\sf r}_n + {\sf r}_n^2  + 4\eta \log
    (2/\eta) \big)\big\}^{1/2} + \big (7.5{\sf r}_n +
    10\eta \sqrt{\log(2/\eta)}\big)\sqrt{\varepsilon} +
    \varepsilon\\
    & \le  \big\{2A\varepsilon \big(\sqrt{2}{\sf r}_n +
    {\sf r}_n^2  + 4\eta \log(2/\eta) \big)
    \big\}^{1/2} + 5.4{\sf r}_n + 7.1\eta \sqrt{\log(2/\eta)}
    + \varepsilon\\
    & \le \big\{2A\varepsilon \big(\sqrt{2}{\sf r}_n +
    4\eta \log(2/\eta) \big) \big\}^{1/2}
    + (\sqrt{A} + 5.4){\sf r}_n + 7.1\eta\sqrt{\log(2/\eta)}
    + \varepsilon\\
    & \le  \sqrt{2\sqrt{2}A{\sf r}_n\varepsilon} +
    2\sqrt{2A\varepsilon\eta\log(2/\eta)} + (\sqrt{A} + 5.4){\sf r}_n +
    7.1\eta\sqrt{\log(2/\eta)} + \varepsilon\\
    & \le  \varepsilon + A{\sf r}_n/\sqrt{2}  +
    2\eta\sqrt{2A\log(2/\eta)} + (\sqrt{A} + 5.4){\sf r}_n +
    7.1\eta\sqrt{\log(2/\eta)} + \varepsilon\\
    & \le (A/\sqrt{2} + \sqrt{A} + 5.4){\sf r}_n+ (7.1 + 2\sqrt{2A})
    \tau \sqrt{\log(2/\tau)} + (9.1 + 2\sqrt{2A})\varepsilon\sqrt{
    \log(2/\varepsilon)}.
\end{align}
These inequalities imply that there exists a universal
constant $\sfC$ such that
\begin{align}\label{eq:18a}
    \big\|\bmuSDR - \bmu^*\big\|_2
    & \le \frac{\sfC \big( {\sf r}_n + \tau\sqrt{\log(2/\tau)}
    + \varepsilon\sqrt{\log(2/\varepsilon)} + {\sf r}_n
    \varepsilon/\sqrt{\tau} \big) \sqrt{L}}{1-2\varepsilon^*}.
\end{align}
We choose
\begin{align}
    \tau = \frac14\bigwedge \frac{\bar{\sf r}_n}{\sqrt{\log_+(2/\bar{\sf r}_n)}}, \qquad\text{with}\qquad
    \bar{\sf r}_n = \frac{\sqrt{\rSigma} +
    \sqrt{2 \log( 2/\delta)}}{\sqrt{n}}.
\end{align}
Note that ${\sf r}_n \le \sqrt{2} \bar{\sf r}_n$ and,
furthermore, $\tau = 1/4$ whenever
$\bar{\sf r}_n \ge 1/2$. Therefore,
${\sf r}_n  \varepsilon/\sqrt{\tau}\le {\sf r}_n + \varepsilon$.
Inserting this value of $\tau$ in \eqref{eq:18a} leads to
\begin{align}
    \big\|\bmuSDR - \bmu^*\big\|_2
    &\le \frac{\sfC \big( {\sf r}_n + \varepsilon\sqrt{\log(2/
    \varepsilon)}\big) \sqrt{L}}{1-2\varepsilon^*}.
\end{align}
where $\sfC$ is a universal constant, the value of which
is not necessarily the same in different places where it
appears. Replacing ${\sf r}_n$ by its expression, and
upper bounding $L$ by $2\log p$,  we arrive at
\begin{align}\label{eq:17}
    \big\|\bmuSDR - \bmu^*\big\|_2
    &\le \frac{\sfC\, \sqrt{\log p}}{1-2\varepsilon^*}
    \bigg( \sqrt{\frac{\rSigma}{n}}+ \varepsilon\sqrt{
    \log(2/\varepsilon)} + \sqrt{\frac{\log(1/\delta)}{n}}
    \bigg).
\end{align}
Note that this inequality holds true on an event of probability
at least $1-5\delta$.

To prove the second part of the theorem, we use \Cref{lem:sigma_p} instead of \Cref{lem:sigmaTr2} in order to bound the term $T_1$. Moreover, in the definitions of ${\sf r}_n$ and $\bar{\sf r}_n$ the effective rank $\rSigma$ is replaced by the dimension $p$. Then, with probability at least $1-2\delta$, we have
\begin{align}
    T_1 &\le \frac{5p + n \eta(8\log(2e/\eta) + 2) + 7\log(2/\delta)}{n(1-\eta)} + 2\frac{\sqrt{p} + \sqrt{\log(2/\delta)}}{\sqrt{n}(1-\sqrt{\eta})} + \\
    &+ \big(
    4{\sf r}_n(1+\sqrt{\eta}) + 10\eta\sqrt{\log(2/\eta)}
    \big)^2 + \varepsilon \\
    &\le (10{\sf r}_n^2 + 15{\sf r}_n + 72\eta\log(2/\eta)) + \big (7.5{\sf r}_n + 10\eta
    \sqrt{\log(2/\eta)}\big)^2 + \varepsilon.
\end{align}
Then, repeating the same steps as for the previous case where the effective rank is used instead of dimension we arrive at the following inequality

\begin{align}
    \sqrt{\varepsilon T_1} &\le \big\{\varepsilon(10{\sf r}_n^2 + 15{\sf r}_n + 72\eta\log(2/\eta))\big\}^{1/2} + \big (7.5{\sf r}_n + 10\eta
    \sqrt{\log(2/\eta)}\big)\sqrt{\varepsilon} + \varepsilon\\
    &\le 12 {\sf r}_n + 16 \tau \sqrt{\log(2/\tau)} + 18 \varepsilon \sqrt{\log(2/\varepsilon)}.
\end{align}

Combining the obtained inequalities, plugging in the values of
$\tau$ and ${\sf r}_n$ and bounding $L$ by $2\log p$ we arrive
at a final bound which reads as
\begin{align}
    \big\|\bmuSDR - \bmu^*\big\|_2
    &\le \frac{156\, \sqrt{2\log p}}{1-2\varepsilon^*}
    \bigg( \sqrt{\frac{2p}{n}}+ \varepsilon\sqrt{
    \log(2/\varepsilon)} + \sqrt{\frac{3\log(2/\delta)}{n}}
    \bigg),
\end{align}
which concludes the proof.

\section{Proof of \Cref{thm:3}}
The proof follows the same steps as that
of \Cref{thm:1}. The assumption $\|\bSigma^{-1/2}\btSigma\bSigma^{-1/2} - \bfI_p\|_{\rm op}
\le \gamma $ gives upper and lower
bounds on the effective rank of $\btSigma$ using that
of $\bSigma$, which we formulate as a separate lemma.
Therefore, the choice of the threshold
parameter $\tilde{t}_\gamma$ stated in \Cref{thm:3} makes
Lemmas \ref{lem:filternoncontamine} and
\ref{lem:filtercontamine} applicable to this case as well.
To bound the operator norm of $\|\hat\bSigma{}^{(\ell)} -
\btSigma \|_{\textup{op}}$ we make use of the assumption
$\|\bSigma^{-1/2}\btSigma\bSigma^{-1/2} - \bfI_p\|_{\rm op}
\le \gamma $ and \Cref{lem:sigmaTr2} using triangle
inequality. We provide the full proof for reader's
convenience.

Before proceeding with the proof we first formulate and prove an auxiliary lemma for bounding the effective rank of $\btSigma$ using that of $\bSigma$.
\begin{lemma}\label{lem:eff-rank}
Let $\bSigma$ and $\btSigma$ be symmetric positive definite matrices such that $\| \bSigma^{-1/2}\btSigma \bSigma^{-1/2} - \bfI_p \|_{\textup{op}} \le \gamma$. Then,
\begin{align}
    \rSigma \cdot \frac{1-\gamma}{1+\gamma} \le \rtSigma \le  \rSigma \cdot \frac{1+\gamma}{1-\gamma},
\end{align}
where by $\rSigma$ we denote the effective rank of matrix $\bSigma$, i.e. $\rSigma = \tr(\bSigma) / \|\bSigma\|_{\textup{op}}$.
\end{lemma}

\begin{proof}[Proof of \Cref{lem:eff-rank}]
    We start with upper- and lower-bounding the operator norm of $\btSigma$.
    Using triangle inequality we have
    \begin{align}
        \big|\|\btSigma\|_{\rm op} - \|\bSigma\|_{\rm op}\big| \le
        \|\btSigma - \bSigma\|_{\rm op} \le  \|\bSigma\|_{\rm op} \cdot \|\bSigma^{-1/2}\btSigma\bSigma^{-1/2} - \bfI_p\|_{\rm op} \le \gamma \|\bSigma\|_{\rm op}.
    \end{align}
    This readily yields
    \begin{align}
        \label{eq:op-up}
        (1-\gamma) \|\bSigma\|_{\rm op} \le
        \|\btSigma\|_{\rm op} \le (1 + \gamma) \|\bSigma\|_{\rm op}.
    \end{align}
    Moreover, for any pair of positive definite matrices
    $\bfA, \bfB \succeq 0$ the following holds $\tr(\bfA) \lambda_1(\bfB)
    \le \tr(\bfA\bfB) \le \tr(\bfA) \cdot \|\bfB\|_{\rm op}$. Hence,
    combining the cyclic property of trace, the trace inequality and the fact that the spectrum of the matrix $\bSigma^{-1/2}\btSigma\bSigma^{-1/2}$ is
    between $1-\gamma$ and $1+\gamma$, we get both the upper and the
    lower bounds for $\tr(\btSigma)$. The upper bound reads as
    \begin{align}
        \tr(\btSigma) &= \tr((\bSigma^{1/2}\bSigma^{-1/2})\btSigma(\bSigma^{-1/2}\bSigma^{1/2})) = \tr(\bSigma (\bSigma^{-1/2}\btSigma\bSigma^{-1/2})) \\
        &\le \| \bSigma^{-1/2}\btSigma\bSigma^{-1/2} \|_{\rm op} \tr(\bSigma) \le (1+\gamma) \tr(\bSigma).\label{eq:tr-up}
    \end{align}
    Similarly, the lower bound can be obtained as follows
    \begin{align}\label{eq:tr-lw}
        \tr(\btSigma) &\ge \lambda_1( \bSigma^{-1/2}\btSigma\bSigma^{-1/2}) \tr(\bSigma) \ge (1-\gamma) \tr(\bSigma).
    \end{align}
Therefore, combining \eqref{eq:op-up} and \eqref{eq:tr-lw} we get the lower bound for $\rtSigma$, while combining \eqref{eq:op-up} and \eqref{eq:tr-up} yields the upper bound, concluding the proof of the lemma.
\end{proof}

Taking $\bfU_L:=\bfV_L$, the \Cref{alg:1} returns
$\bmuSDR=\sum_{\ell=0}^L\hat\bmu^{(\ell)}$ with
$\hat\bmu^{(\ell)}\in\mathscr U_\ell={\rm Im}
(\bfV_\ell\bfU_\ell^\top)$ for every $\ell\in
\{0,\ldots,L\}$ where the two-by-two orthogonal
subspaces $\mathscr U_0,\ldots,\mathscr U_L$ span
the whole space $\mathbb{R}^p$. This allows us to
decompose the risk:
\begin{align}
    \big\|\bmuSDR - \bmu^*\big\|_2^2 &= \sum_{\ell=0}^L
    \big\|\hat\bmu^{(\ell)} - \Proj_{\mathscr U_\ell}
    \bmu^* \big\|_2^2 = \sum_{\ell=0}^L \big\|
    \Proj_{\mathscr U_\ell}(\barx_{\ell} - \bmu^*)
    \big\|_2^2\\
    &=\sum_{\ell=0}^L\big\| \Proj_\ell (\bfV_{\ell}^\top
    \barx_{\ell}  - \bfV_{\ell}^\top\bmu^*)\big\|_2^2,
\end{align}
where $\Proj_\ell:=\bfU_{\ell}^\top\bfU_\ell$ is the projection matrix projecting onto the subspace of $\mathbb{R}^{p_{\ell}}$ spanned by the bottom $p_\ell - {p_{\ell+1}}$ eigenvectors of $\bfV_\ell^\top(\hat\bSigma{}^{(\ell)} - \btSigma)\bfV_\ell$ for $\ell=0,\ldots,L$ with the convention that $p_{L+1}=0$.

For $\ell\in\{0,\ldots,L-1\}$, we intend to apply
\Cref{prop:projectedmean} to $\bZ_i = \bfV_{\ell}^\top\bX_i$
and $\bmu^Z = \bfV_{\ell}^\top\bmu^*$ in order to upper bound
the error term $\Err_\ell := \|\Proj_\ell(\bfV_{\ell}^\top
\barx_{\ell}  - \bfV_{\ell}^\top \bmu^* )\big\|_2$. Using
the inequalities
\begin{align}
    \|\bfV^\top
    (\hat\bSigma{}^{(\ell)} - \btSigma)\bfV\|_{\rm op} \le
    \|\hat\bSigma{}^{(\ell)} - \btSigma\|_{\rm op},
    \quad
    \lambda_{p_\ell} (\bfV^\top \btSigma \bfV)
    \le \lambda_{p} (\btSigma),
    \quad
    \lambda_{1} (\bfV^\top \btSigma \bfV)
    \ge \lambda_{1} (\btSigma)
\end{align}
that are true for every orthogonal matrix $\bV$,
and keeping in mind the definition of $\Proj_\ell$,
we get
\begin{align}
    \Err_\ell &\le
    \bigg\{2\omega_\calO \|
    \hat\bSigma{}^{(\ell)} - \btSigma\|_{\rm op}
    + \frac{\omega_\calO^2}{1-\omega_\calO}
    \bigg((\lambda_{p}-\lambda_1)(\btSigma) + \frac{\delta_\ell^2}{p_{\ell+1}}\bigg)
    \bigg\}^{1/2} + \big\|   \Proj_\ell\bfV_\ell^\top
    \barxi_{\calS^{(\ell)}_\calI }\big\|_2,
\end{align}
where we have used the notation
\begin{align}
    \omega_{\calO} = \max_{\ell} \frac{|\calS^{(\ell)}\cap\calO|}
    {|\calS^{(\ell)}|},\qquad
    \bxi_i = \bX_i -\bmu^*
\end{align}
and
$\delta_\ell = \inf_{\bmu}
\max_{i\in\calS^{(\ell)}} \|\bfV_\ell^\top (\bX_i -
\bmu)\|_2$. Note that when $\calO$ and $(\calS^{(\ell)
}_\calI)^c$ are of cardinality less than $n\varepsilon$
and $n(\varepsilon + \tau)$, respectively, we have
$\omega_{\calO} \le {\varepsilon}/{(1-\tau)}$ and
\begin{align}
    \frac{\omega_{\calO}}{1-\omega_{\calO}}
    \le \frac{\varepsilon}{1- \varepsilon - \tau}.
\end{align}
Since $\sfC_\gamma \rtSigma \ge \rSigma$ (by \Cref{lem:eff-rank})
then \Cref{lem:filtercontamine} holds for the new threshold
$\tilde{t}_\gamma$ as well. We set $\eta := \varepsilon +
\tau\le 3/4$ and apply \Cref{lem:filtercontamine} to infer
that $\omega_{\calO} \leq \varepsilon/(1 - \eta)\le 4\varepsilon$
is true with probability at least $1-\delta$. Furthermore, we
know that $\delta_\ell\leq\max_{i\in\calS^{(\ell)}}
\|\bfV_\ell^\top\bX_i - \bar\bmu^{(\ell)}\|_2\le
\tilde{t}_\gamma\sqrt{p_\ell}$. This yields
\begin{align}
    \Err_\ell &\leq \Big\{8\varepsilon
    \|\hat\bSigma{}^{(\ell)} - \btSigma\|_{\rm op}
    + 16\varepsilon^2 \Big((\lambda_{p}-\lambda_1) (\btSigma)
    + \frac{\tilde{t}^2_\gamma p_{\ell}}{p_{\ell+1}}\Big)\Big\}^{1/2}
    + \big\|\Proj_{\mathscr U_\ell}\barxi_{\calS^{(\ell)
    }_\calI}\big\|_2.
\end{align}
Let us introduce the shorthand
$\widetilde{T}_1 =  \max_{\ell\in[L]}
    \|\hat\bSigma{}^{(\ell)} - \btSigma\|_{\rm op}
    + \varepsilon (\lambda_{p}-\lambda_1) (\btSigma)$.
This leads to
\begin{align}\label{Err:l0}
    \Err_\ell &\leq \Big\{8\varepsilon \widetilde{T}_1 +
    \frac{16\varepsilon^2 \tilde{t}_\gamma^2p_{\ell}}{p_{\ell+1}}
    \Big\}^{1/2} + \| \Proj_{\mathscr U_\ell}
    \barxi_{\calS^{(\ell) }_\calI}\big\|_2.
\end{align}
For the last error term, since $p_L=1$ we have by \Cref{lem:lowdim}
\begin{align}
  \Err_L&\le \big\|\Proj_{\mathscr U_L} \bar\bxi_{\calS_\calI^{(L)}} \big\|_2 + \frac{n
    \varepsilon (\tilde{t}_\gamma\sqrt{p_L} + \big\|\Proj_{\mathscr U_L}\bmu^* - \med_{\mathscr U_L}\big\|_2)}{
    |\mathcal S^{(L)}|}\\
    &\le \big\|\Proj_{\mathscr U_L} \bar\bxi_{\calS_\calI^{(L)}} \big\|_2 + \frac{
    \varepsilon \tilde{t}_\gamma + \varepsilon\big\|\Proj_{\mathscr U_L}\bmu^* - \med_{\mathscr U_L}\big\|_2}{1-\eta}\\
    &\le \big\|\Proj_{\mathscr U_L} \bar\bxi_{\calS_\calI^{(L)}} \big\|_2 +
    4\varepsilon \tilde{t}_\gamma + 4\varepsilon\big\|\Proj_{\mathscr U_L}\bmu^* -\med_{\mathscr U_L}\big\|_2.\label{Err:L0}
\end{align}
Combining \eqref{Err:l0}, \eqref{Err:L0}, inequality
$p_\ell\le e p_{\ell+1}$, as well as
the Minkowski inequality, we get
\begin{align}
    \big\|\bmu^* &- \bmuSDR\big\|_2  = \bigg\{\sum_{\ell=0}^L
    \Err_\ell^2\bigg\}^{1/2}\\
    &\le \bigg\{ 8\varepsilon L ( \widetilde{T}_1 + e \varepsilon
    t^2) + 16\varepsilon^2\big(\tilde{t}_\gamma + \big\|\Proj_{\mathscr
    U_L}\bmu^* -\med_{\mathscr U_L}\big\|_2\big)^2\bigg\}^{1/2}
    + \bigg\{\sum_{\ell=0}^L \| \Proj_{\mathscr U_\ell}
    \barxi_{\calS^{(\ell) }_\calI}\big\|_2^2\bigg\}^{1/2}\\
    &\le 2\sqrt{2\varepsilon L \widetilde{T}_1} + 9\varepsilon \tilde{t}_\gamma
    \sqrt{L} + 4\varepsilon \big\|\Proj_{\mathscr
    U_L}\bmu^* -\med_{\mathscr U_L}\big\|_2 +
    \bigg\{\sum_{\ell=0}^L \| \Proj_{\mathscr U_\ell}
    \barxi_{\calS^{(\ell) }_\calI}\big\|_2^2\bigg\}^{1/2}\!\!
    \!\!.
    \label{eq:14a}
\end{align}
To ease notation, 
let us set
\begin{align}
    {\sf r}_n = \Big(\frac{2\rSigma+3\log(2/\delta)}{n}\Big)^{1/2}.
\end{align}
In view of \Cref{lem:mean}, with
probability at least $1-\delta$, we have
\begin{align}
    \bigg\{\sum_{\ell=0}^L \| \Proj_{\mathscr U_\ell}
    \barxi_{\calS^{(\ell) }_\calI}\big\|_2^2\bigg\}^{1/2}
    &\le
    \bigg\{\sum_{\ell=0}^L \bigg( \frac{\|\Proj_{\mathscr U_\ell}\bar\bxi_n
    \|_2}{1 - \eta}  +  \frac{{\sf r}_n\sqrt{\eta} + \eta\sqrt{3 \log(2e/\eta)}}{1-\eta}\bigg)^2\bigg\}^{1/2}\\
    &\le \bigg\{\sum_{\ell=0}^L \big( 4\|\Proj_{\mathscr U_\ell}
    \bar\bxi_n \|_2  +  4{\sf r}_n\sqrt{\eta} + 10\eta\sqrt{
    \log(2/\eta)}\big)^2\bigg\}^{1/2}\\
    &\le 4\|\bar\bxi_n \|_2 + 4{\sf r}_n\sqrt{\eta L} +
    10\eta\sqrt{L\log(2/\eta)}.
\end{align}
Since  $\|\barxi_n\|_2^2$ has the same distribution as $\frac1n\sum_{j=1}^p\lambda_j(\bSigma)\gamma_j^2$, where $\gamma_1,\dots,\gamma_p$ are i.i.d. standard Gaussian, by \cite[Lemma 8]{Comminges} we have
\begin{align}
    \|\barxi_n\|_2^2\le\frac{2\rSigma+3\log(2/\delta)}{n}
    ={\sf r}_n^2
\end{align}
with probability at least $1-\delta$. Therefore, with
probability at least $1-2\delta$,
\begin{align}
    \bigg(\sum_{\ell=0}^L \| \Proj_{\mathscr U_\ell}
    \barxi_{\calS^{(\ell) }_\calI}\big\|_2^2\bigg)^{1/2}
    &\le  4{\sf r}_n \big(1 + \sqrt{L\eta}\big) + 10\eta
    \sqrt{L\log(2/\eta)}.\label{eq:14bisa}
\end{align}
Next, \Cref{lem:filtercontamine} and the fact that
$p_L = \dim(\mathscr U_L) = 1$  imply that with probability at least $1-\delta$
\begin{align}
    \big\|\Proj_{\mathscr
    U_L}\bmu^* -\med_{\mathscr U_L}\big\|_2 \le
    \frac{2(1+ \sqrt2{\sf r}_n)}{1-2\varepsilon}.\label{eq:15a}
\end{align}
Recall that we have chosen $\tilde{t}_\gamma$ in such a way that
\begin{align}
    \tilde{t}_\gamma \le \frac{3(1 + {\sfC_\gamma}\sqrt2{\sf r}_n/\sqrt{\tau})}{1-2\varepsilon^*}
        +  1.6\sqrt{\log(2/{{\tau}})}.\label{eq:16a}
\end{align}
Combining \eqref{eq:14a}, \eqref{eq:14bisa}, \eqref{eq:15a} and \eqref{eq:16a},
we arrive at the inequality
\begin{align}
    \big\|\bmuSDR - \bmu^*\big\|_2
    &\le 2\sqrt{2\varepsilon L \widetilde{T}_1} + 9\varepsilon \tilde{t}_\gamma
    \sqrt{L} +  \frac{8\varepsilon(1+{\sfC_\gamma}\sqrt{2}{\sf r}_n)}{1-2\varepsilon}
    + 4{\sf r}_n \big(1 + \sqrt{L\eta}\big) + 10\eta
    \sqrt{L\log(2/\eta)}\\
    &\le 2\sqrt{2\varepsilon L \widetilde{T}_1} +
    \frac{27\varepsilon \sqrt{L}(1 + {\sfC_\gamma}\sqrt2{\sf r}_n/\sqrt{\tau})}{1 -
    2\varepsilon^*} + 14.4\varepsilon \sqrt{L\log(2/\tau)} \\
    &\quad +  \frac{8\varepsilon(1+{\sfC_\gamma}\sqrt{2}{\sf r}_n)}{1-2\varepsilon}
    + 4{\sf r}_n \big(1 + \sqrt{L\eta}\big) + 10\eta
    \sqrt{L\log(2/\eta)}
\end{align}
that holds with probability at least $1-3\delta$. In the
upper bound obtained above, only the term $\widetilde{T}_1$ remains
random. To bound $\widetilde{T}_1$ we first apply a triangle inequality then use \Cref{lem:sigmaTr2}. It implies that with probability
at least $1-2\delta$, we have
\begin{align}
    \widetilde{T}_1 &\le  A\frac{\sqrt{n\rSigma} + \rSigma + n\eta
    \log(2e/\eta) + 2\log(1/\delta)}{n(1-\eta)} + \big(
    4{\sf r}_n(1+\sqrt{\eta}) + 10\eta\sqrt{\log(2/\eta)}
    \big)^2 + (1+\gamma)\varepsilon + \gamma\\
    &\le 2A\big(\sqrt{2}{\sf r}_n + {\sf r}_n^2  + 4\eta
    \log(2/\eta)\big) + \big (7.5{\sf r}_n + 10\eta
    \sqrt{\log(2/\eta)}\big)^2 + \varepsilon + 2\gamma.
\end{align}

Consequently,
\begin{align}
    \sqrt{\varepsilon \widetilde{T}_1} &\le  \big\{2A\varepsilon
    \big(\sqrt{2}{\sf r}_n + {\sf r}_n^2  + 4\eta \log
    (2/\eta) \big)\big\}^{1/2} + \big (7.5{\sf r}_n +
    10\eta \sqrt{\log(2/\eta)}\big)\sqrt{\varepsilon} +
    \varepsilon + \sqrt{2\varepsilon\gamma}\\
    & \le  \big\{2A\varepsilon \big(\sqrt{2}{\sf r}_n +
    {\sf r}_n^2  + 4\eta \log(2/\eta) \big)
    \big\}^{1/2} + 5.4{\sf r}_n + 7.1\eta \sqrt{\log(2/\eta)}
    + \varepsilon + \sqrt{2\varepsilon\gamma}\\
    & \le \big\{2A\varepsilon \big(\sqrt{2}{\sf r}_n +
    4\eta \log(2/\eta) \big) \big\}^{1/2}
    + (\sqrt{A} + 5.4){\sf r}_n + 7.1\eta\sqrt{\log(2/\eta)}
    + \varepsilon+ \sqrt{2\varepsilon\gamma}\\
    & \le  \sqrt{2\sqrt{2}A{\sf r}_n\varepsilon} + 2\sqrt{2A\varepsilon\eta\log(2/\eta)} + (\sqrt{A} + 5.4){\sf r}_n +
    7.1\eta\sqrt{\log(2/\eta)} + \varepsilon+ \sqrt{2\varepsilon\gamma}\\
    & \le  \varepsilon + A{\sf r}_n/\sqrt{2}  + 2\eta\sqrt{2A\log(2/\eta)} + (\sqrt{A} + 5.4){\sf r}_n +
    7.1\eta\sqrt{\log(2/\eta)} + \varepsilon+ \sqrt{2\varepsilon\gamma}\\
    & \le (A/\sqrt{2} + \sqrt{A} + 5.4){\sf r}_n+ (7.1 + 2\sqrt{2A})
    \tau \sqrt{\log(2/\tau)} + (9.1 + 2\sqrt{2A})\varepsilon\sqrt{
    \log(2/\varepsilon)}+ \sqrt{2\varepsilon\gamma}.
\end{align}
These inequalities imply that there exists a universal
constant $\sfC$ such that
\begin{align}\label{eq:18b}
    \big\|\bmuSDR - \bmu^*\big\|_2
    & \le \frac{\sfC \big({\sfC_\gamma}{\sf r}_n + \tau\sqrt{\log(2/\tau)}
    + \varepsilon\sqrt{\log(2/\varepsilon)} + {\sf r}_n
    \varepsilon/\sqrt{\tau} + \sqrt{\varepsilon\gamma} \big) \sqrt{L}}{1-2\varepsilon^*}.
\end{align}
Let us denote $\log_+(x) = \max\{0, \log(x)\}$ the positive part of logarithm, then we choose
\begin{align}
    \tau = \frac14\bigwedge \frac{\tilde{\sf r}_n}{\sqrt{\log_+(2/\tilde{\sf r}_n)}}, \qquad\text{with}\qquad
    \tilde{\sf r}_n = \frac{\sqrt{\sfC_\gamma\rtSigma} +
    \sqrt{2 \log( 2/\delta)}}{\sqrt{n}}.
\end{align}
Note that ${\sf r}_n \le \sqrt{2} \tilde{\sf r}_n$ and,
furthermore, $\tau = 1/4$ whenever
$\tilde{\sf r}_n \ge 1/2$. Therefore,
${\sf r}_n  \varepsilon/\sqrt{\tau}\le {\sf r}_n + \varepsilon$.
Inserting this value of $\tau$ in \eqref{eq:18b} leads to
\begin{align}
    \big\|\bmuSDR - \bmu^*\big\|_2
    &\le \frac{\sfC \big( {\sfC_\gamma}{\sf r}_n + \varepsilon\sqrt{\log(2/
    \varepsilon)} + \sqrt{\varepsilon\gamma}\big) \sqrt{L}}{1-2\varepsilon^*}.
\end{align}
where $\sfC$ is a universal constant, the value of which
is not necessarily the same in different places where it
appears. Replacing ${\sf r}_n$ by its expression,
upper bounding $L$ by $2\log p$, and using the fact that $\sfC_\gamma \le 3$ for $\gamma\in(0, 1/2]$ we arrive at
\begin{align}\label{eq:17}
    \big\|\bmuSDR - \bmu^*\big\|_2
    &\le \frac{\sfC\, \sqrt{\log p}}{1-2\varepsilon^*}
    \bigg( \sqrt{\frac{\rSigma}{n}}+ \sqrt{\frac{\log(1/\delta)}{n}} + \varepsilon\sqrt{
    \log(2/\varepsilon)} + \sqrt{\varepsilon\gamma}
    \bigg).
\end{align}
Note that this inequality holds true on an event of probability
at least $1-5\delta$.

\section{Extension to Sub-Gaussian distributions}
This section is devoted to the proof of \Cref{thm:subG}, which is an extension of \Cref{thm:1} to the case when the $1-\varepsilon$ portion of observations are sub-Gaussian. First, we formulate some technical lemmas necessary for the proof of \Cref{thm:subG} postponing the full proof to the end of the present section.

Recall that a random vector $\bzeta$ with zero
mean and identity covariance matrix is sub-Gaussian with variance proxy ${\mathfrak s} > 0$,
$\bzeta \sim \textup{SG}_p({\mathfrak s})$, if
\begin{align}
    \mathbb{E}[e^{\bv^\top \bzeta}] \le \exp\Big\{\frac{{\mathfrak s}}{2}\|\bv \|^2\Big\}, \quad \forall \bv \in \RR^p.
\end{align}
The concentration inequality for sub-Gaussian vectors is a well-known fact (see, e.g. \citep{RigolletHutter}, Theorem 1.19) that if $\bzeta \sim \textup{SG}_p({\mathfrak s})$ then for all $\delta \in (0,1)$, it holds
\begin{align}\label{eq:subG}
    \prob\big(\| \bzeta \|_2 \le 4\sqrt{p{\mathfrak s}} + \sqrt{8{\mathfrak s}\log(1/\delta)}\big) \ge 1 - \delta.
\end{align}
In the definition of SGAC$(\bmu^*,\bSigma,\mathfrak{s},\varepsilon)$ we assume that the $1-\varepsilon$ portion of the data $\{\bX_i\}_{i=1}^n$ are sub-Gaussian, that is $\bX_i = \bmu^* + \bSigma^{1/2}\bzeta_i$ for all $i \in \mathcal{I}$, where the set $\mathcal{I}$ is of cardinality at least $(1-\varepsilon)n$. Denote $\bxi_i = \bSigma^{1/2}\bzeta_i$ for all $i = 1, \dots, n$ and assume that $\| \bSigma\|_{\textup{op}} = 1$.

First, we show that with the choice of threshold parameter the analogous to \Cref{lem:geommed}, \Cref{lem:filternoncontamine}, \Cref{lem:filtercontamine} lemmas hold true. Notice that all three lemmas are using the concentration bound for the operator norm of (sub)-Gaussian vectors. In the case of Gaussian vectors we make use of \cite[Corollary 5.35]{Vershynin2012IntroductionTT}, while the analogous result for the sub-Gaussian distributions is also known \cite[Theorem 5.39]{Vershynin2012IntroductionTT}. For the readers convenience the latter is formulated in \Cref{lem:vershynin:bis}.

\begin{lemma}\label{lem:11}
    \it{
    Let $ J \subset \{1,\ldots,n\}$ be a subset of cardinality $m$.
    For every $\delta\in(0,1)$, it holds that
    \begin{align}
        \mathbf P\bigg(\bigg\|\sum_{j\in J} \bxi_j\bigg\|_2
        \le 4\sqrt{pm{\mathfrak s}} + \sqrt{8m{\mathfrak s}\log(1/\delta)}
        \bigg)\ge 1-\delta.
    \end{align}
    }
    \end{lemma}

    \begin{proof}[Proof of Lemma \ref{lem:11}]
    Without loss of generality, we assume that $ J = \{1,\ldots,m\}$.
    On the one hand,  $\|\bSigma\|_{\textup{op}}=1$ implies that
    \begin{align}
        \bigg\|\sum_{i=1}^m \bxi_i\bigg\|_2 \le
        \bigg\|\sum_{i=1}^m \bzeta_i\bigg\|_2.
    \end{align}
    On the other hand, $\bzeta_1+\ldots+\bzeta_m\sim \textup{SG}_p(m{\mathfrak s})$.
    Applying inequality \eqref{eq:subG} to this random variable
    yields the desired result.
    \end{proof}

We now state the versions of \Cref{lem:geommed}
and \Cref{lem:filtercontamine} that are valid
in the setting of sub-Gaussian vectors. Notice
that the only difference is in the choice of
the threshold; thanks to \Cref{lem:vershynin:bis}, the
threshold now includes the universal constant
$\sfC_0$ and the variance proxy $\mathfrak{s}$.
The proofs of these lemmas are omitted, since
they are the same as in the Gaussian case
presented in \Cref{ssec:7.2} (except for
bounding the operator norm of a matrix with
sub-Gaussian columns we use \Cref{lem:vershynin:bis}).
	
\begin{lemma}\label{lem:geommed_subG}
    With probability at least $1-\delta$, for all
    linear subspaces $V\subset \mathbb R^p$, we have
    \begin{align}
        \frac{\|\bmuGMV - \ProjV\bmu^*\|_2}{\sqrt{\dim(V)}}
        &\leq \frac{2 \sqrt{\|\bSigma\|_{\rm op}}}{1-2\varepsilon}
        \bigg(1 + \frac{\sfC_0\big({\mathfrak{s}}\sqrt{p} + 2{\mathfrak{s}}\sqrt{\log(1/\delta)}
        \big)}{\sqrt{n}}\bigg),
    \end{align}
    where the constant $\sfC_0$ is the
    same constant as in \Cref{lem:vershynin:bis}.
\end{lemma}

\begin{lemma}\label{lem:filternoncontamine_subG}
Let $\tau$ and $\delta$ be two numbers from $(0,1)$. Define
\begin{align}
    z = 1 + \frac{\sfC_0\mathfrak{s}(\sqrt{p} + \sqrt{2 \log(1/\delta)})
        }{\sqrt{n\tau}} + \sfC_0{\mathfrak{s}}\sqrt{2 + 2\log\big(1/\tau\big)}
\end{align}
with the same constant $\sfC_0$ as in \Cref{lem:vershynin:bis}. Then, with probability at least $1-\delta$, we have
\begin{align}
    \sup_{V} \sum_{i=1}^n \indic\big( \|\ProjV\bxi_i\|_2^2
    > z^2 \dim(V) \big) \le n\tau,
\end{align}
where the supremum is over all linear subspaces $V$ of
$\mathbb R^p$.
\end{lemma}

\begin{lemma}[\cite{KoltLoun}, Theorem 9] \label{lem:subGopnorm}
	There is a constant $A_3>0$ depending only on the variance
	proxy $\tau$ such that for every pair of integers $n\ge 1$
	and $p\ge 1$, we have
	\begin{align}
		\bfP\bigg(\|\bzeta_{1:n}\bzeta_{1:n}^\top-n\bSigma\|_{
		\textup{op}} \ge A_3\Big(\sqrt{(p+t)n} + p + t\Big)\bigg)
		\le e^{-t},\quad \forall t\ge 1.
	\end{align}
\end{lemma}

\begin{lemma}\label{lem:sigma_subG}
There exists a positive constant $A$ such that, for any positive
integer $m \le n$ and any $t \ge 1$, with probability at least $1 - 2e^{-t}$, the inequality
\begin{align}
    \|\hat{\bSigma}_{\calS} - \bSigma\|_{\rm op}\leq
    A\frac{\sqrt{np} + p + m\log(2ne/m)+2t}{n-m} +
    \big\|\barxi_\calS\big\|_2^2
\end{align}
is satisfied for every $\calS\subset [n]$ of cardinality
$\geq n-m$.
\end{lemma}

\begin{proof}
    The proof of this lemma is similar to the proof of
    \Cref{lem:sigmaTr2} with the only difference that
    instead of Theorems 4 and 5 from \citep{KoltLoun}
    we now use  \Cref{lem:subGopnorm}.
\end{proof}

\begin{lemma}\label{lem:mean_subG}
For any positive integer $m\le n$ and any $t>0$,
with probability at least $1 - e^{-t}$, we have
\begin{align}
    \max_{|S|\ge n-m}\Big\|\frac1{|S|}\sum_{i\in S}
    \Proj\,\bxi_i\Big\|_2 \leq \frac{n\|\Proj\bar\bxi_n
    \|_2}{n-m}  +  \frac{\sqrt{m{\mathfrak s}}(4\sqrt{p}+2\sqrt{2t}) +
    2m\sqrt{2{\mathfrak s}\log(2ne/m)}}{n-m}.
\end{align}
\end{lemma}

\begin{proof}
The proof of this theorem is similar to that of \Cref{lem:mean}, with the only exception that now we need to bound the maximum of a norm of a sum of at most $m$ sub-Gaussian vectors, where the maximum is taken over all subsets of $[n]$ of size at most $m$. Since, each sub-Gaussian vector has a variance proxy $\mathfrak{s}$ then using \Cref{lem:11} along with union bound, we have
\begin{align}
    \bfP\Big(\max_{|J|\leq m} \Big\|\sum_{i \in J}
    \bxi_i \Big\|_2 \geq \sqrt{m\mathfrak{s}}(4\sqrt{p} + t_m)\Big)
    &\leq \sum_{l=1}^{m} \binom{n}{l} \bfP\Big( \Big\|
    \sum_{i=1}^l \bxi_i\Big\|_2 \geq \sqrt{m\mathfrak{s}}(4\sqrt{p} + t_m)\Big)\\
    &\leq \sum_{l=1}^{m} \Big(\frac{ne}{l}\Big)^l \bfP\Big(
    \Big\|\sum_{i=1}^l \bxi_i\Big\|_2 \geq \sqrt{l\mathfrak{s}}(4\sqrt{p} + t_l)
    \Big)\\
    &\leq \sum_{l=1}^{m} \Big(\frac{ne}{l}\Big)^l e^{-t_l/3}
    \leq e^{-t}.
\end{align}

Therefore, we obtain that with probability at least $1-e^{-t}$ the inequality
\begin{align}
    \max_{|J|\le m} \Big\| \sum_{i\in J}
    \bxi_i \Big\|_2 \le \sqrt{m{\mathfrak s}}(4\sqrt{p}+2\sqrt{2t}) +
    2m\sqrt{2{\mathfrak s}\log(2ne/m)}
\end{align}
holds. Then, combining with
\begin{align}
    \frac{1}{|S|}\Big\| \sum_{i\in S}\Proj\bxi_i\Big\|_2 \le \frac{n\|\Proj\bar\bxi_n\|_2}{n - m} +
    \frac{1}{n-m}\max_{|J|\le m} \Big\| \sum_{i\in J}
    \bxi_i \Big\|_2.
\end{align}
yields the desired result.

\end{proof}

\subsection{Proof of \Cref{thm:subG}}
All the ingredients provided, we can now compile the complete proof of \Cref{thm:subG}.

Taking $\bfU_L:=\bfV_L$, the algorithm detailed in \eqref{alg:3} returns $\bmuSDR=\sum_{\ell=0}^L\hat\bmu^{(\ell)}$ with $\hat\bmu^{(\ell)}\in\mathscr U_\ell={\rm Im}(\bfV_\ell\bfU_\ell^\top)$ for every $\ell\in\{0,\ldots,L\}$ where the two-by-two orthogonal subspaces $\mathscr U_0,\ldots,\mathscr U_L$ span the whole space $\mathbb{R}^p$.
This allows us to decompose the risk:
\begin{align}
    \big\|\bmuSDR - \bmu^*\big\|_2^2 &
    = \sum_{\ell=0}^L \big\|
    \Proj_{\mathscr U_\ell}(\barx_{\ell} - \bmu^*)
    \big\|_2^2
    =\sum_{\ell=0}^L\big\| \Proj_\ell (\bfV_{\ell}^\top
    \barx_{\ell}  - \bfV_{\ell}^\top\bmu^*)\big\|_2^2,
\end{align}
where $\Proj_\ell:=\bfU_{\ell}^\top\bfU_\ell$ is the projection matrix projecting onto the subspace of $\mathbb{R}^{p_{\ell}}$ spanned by the bottom $p_\ell - {p_{\ell+1}}$ eigenvectors of $\bfV_\ell^\top(\hat\bSigma{}^{(\ell)} - \bSigma)\bfV_\ell$ for $\ell=0,\ldots,L$ with the convention that $p_{L+1}=0$.

For $\ell\in\{0,\ldots,L-1\}$, we intend to apply
\Cref{prop:projectedmean} to $\bZ_i = \bfV_{\ell}^\top\bX_i$
and $\bmu^Z = \bfV_{\ell}^\top\bmu^*$ in order to upper bound
the error term $\Err_\ell := \|\Proj_\ell(\bfV_{\ell}^\top
\barx_{\ell}  - \bfV_{\ell}^\top \bmu^* )\big\|_2$. Using
the inequalities
\begin{align}
    \|\bfV^\top
    (\hat\bSigma{}^{(\ell)} - \bSigma)\bfV\|_{\rm op} \le
    \|\hat\bSigma{}^{(\ell)} - \bSigma\|_{\rm op},
    \quad
    \lambda_{p_\ell} (\bfV^\top \bSigma \bfV)
    \le \lambda_{p} (\bSigma),
    \quad
    \lambda_{1} (\bfV^\top \bSigma \bfV)
    \ge \lambda_{1} (\bSigma)
\end{align}
that are true for every orthogonal matrix $\bV$,
and keeping in mind the definition of $\Proj_\ell$,
we get
\begin{align}
    \Err_\ell &\le
    \bigg\{2\omega_\calO \|
    \hat\bSigma{}^{(\ell)} - \bSigma\|_{\rm op}
    + \frac{\omega_\calO^2}{1-\omega_\calO}
    \bigg((\lambda_{p}-\lambda_1)(\bSigma) + \frac{\delta_\ell^2}{p_{\ell+1}}\bigg)
    \bigg\}^{1/2} + \big\|   \Proj_\ell\bfV_\ell^\top
    \barxi_{\calS^{(\ell)}_\calI }\big\|_2,
\end{align}
where we have used the notation
\begin{align}
    \omega_{\calO} = \max_{\ell} \frac{|\calS^{(\ell)}\cap\calO|}
    {|\calS^{(\ell)}|},\qquad
    \bxi_i = \bX_i -\bmu^*
\end{align}
and
$\delta_\ell = \inf_{\bmu}
\max_{i\in\calS^{(\ell)}} \|\bfV_\ell^\top (\bX_i -
\bmu)\|_2$. Note that when $\calO$ and $(\calS^{(\ell)
}_\calI)^c$ are of cardinality less than $n\varepsilon$
and $n(\varepsilon + \tau)$, respectively, we have
$\omega_{\calO} \le {\varepsilon}/{(1- \tau)}$ and
$\frac{\omega_{\calO}}{1-\omega_{\calO}}
    \le \frac{\varepsilon}{1- \varepsilon - \tau}$.

We set $\eta := \varepsilon + \tau\le 3/4$ and apply
\Cref{lem:filternoncontamine_subG} to infer that $\omega_{\calO} \leq
\varepsilon/(1 - \eta)\le 4\varepsilon$ is true with probability
at least $1-\delta$. Furthermore, we know that $\delta_\ell\leq\max_{i\in\calS^{(\ell)}}\|\bfV_\ell^\top\bX_i - \bar\bmu^{(\ell)}\|_2\le t\sqrt{p_\ell}$. This yields
\begin{align}
    \Err_\ell &\leq \Big\{8\varepsilon
    \|\hat\bSigma{}^{(\ell)} - \bSigma\|_{\rm op}
    + 16\varepsilon^2 \Big((\lambda_{p}-\lambda_1) (\bSigma)
    + \frac{t^2p_{\ell}}{p_{\ell+1}}\Big)\Big\}^{1/2}
    + \big\|\Proj_{\mathscr U_\ell}\barxi_{\calS^{(\ell)
    }_\calI}\big\|_2.
\end{align}
Let us introduce the shorthand
$T_1 =  \max_{\ell\in[L]}
    \|\hat\bSigma{}^{(\ell)} - \bSigma\|_{\rm op}
    + \varepsilon (\lambda_{p}-\lambda_1) (\bSigma)$.
This leads to
\begin{align}\label{Err:ll}
    \Err_\ell &\leq \Big\{8\varepsilon T_1 +
    \frac{16\varepsilon^2 t^2p_{\ell}}{p_{\ell+1}}
    \Big\}^{1/2} + \| \Proj_{\mathscr U_\ell}
    \barxi_{\calS^{(\ell) }_\calI}\big\|_2.
\end{align}
For the last error term, since $p_L=1$ then, by the combination of \Cref{lem:geommed_subG}
and \Cref{lem:filternoncontamine_subG}, we have
\begin{align}
  \Err_L&\le \big\|\Proj_{\mathscr U_L} \bar\bxi_{\calS_\calI^{(L)}} \big\|_2 + \frac{n
    \varepsilon (t\sqrt{p_L} + \big\|\Proj_{\mathscr U_L}\bmu^* - \med_{\mathscr U_L}\big\|_2)}{
    |\mathcal S^{(L)}|}\\
    &\le \big\|\Proj_{\mathscr U_L} \bar\bxi_{\calS_\calI^{(L)}} \big\|_2 + \frac{
    \varepsilon t + \varepsilon\big\|\Proj_{\mathscr U_L}\bmu^* - \med_{\mathscr U_L}\big\|_2}{1-\eta}\\
    &\le \big\|\Proj_{\mathscr U_L} \bar\bxi_{\calS_\calI^{(L)}} \big\|_2 +
    4\varepsilon t + 4\varepsilon\big\|\Proj_{\mathscr U_L}\bmu^* -\med_{\mathscr U_L}\big\|_2.\label{Err:LL}
\end{align}
Combining \eqref{Err:ll}, \eqref{Err:LL}, inequality
$p_\ell\le e p_{\ell+1}$, as well as
the Minkowski inequality, we get
\begin{align}
    \big\|\bmu^* &- \bmuSDR\big\|_2  = \bigg\{\sum_{\ell=0}^L
    \Err_\ell^2\bigg\}^{1/2}\\
    &\le \bigg\{ 8\varepsilon L ( T_1 + e \varepsilon
    t^2) + 16\varepsilon^2\big(t + \big\|\Proj_{\mathscr
    U_L}\bmu^* -\med_{\mathscr U_L}\big\|_2\big)^2\bigg\}^{1/2}
    + \bigg\{\sum_{\ell=0}^L \| \Proj_{\mathscr U_\ell}
    \barxi_{\calS^{(\ell) }_\calI}\big\|_2^2\bigg\}^{1/2}\\
    &\le 2\sqrt{2\varepsilon L T_1} + 9\varepsilon t
    \sqrt{L} + 4\varepsilon \big\|\Proj_{\mathscr
    U_L}\bmu^* -\med_{\mathscr U_L}\big\|_2 +
    \bigg\{\sum_{\ell=0}^L \| \Proj_{\mathscr U_\ell}
    \barxi_{\calS^{(\ell) }_\calI}\big\|_2^2\bigg\}^{1/2}\!\!
    \!\!.
    \label{eq:14c}
\end{align}
To ease notation, 
let us set
\begin{align}
    {\sf r}_{n, \mathfrak{s}} = \frac{4\sqrt{\mathfrak{s}}\big(\sqrt{p}+2\sqrt{\log(2/\delta)}\big)}{\sqrt{n}}.
\end{align}

In view of \Cref{lem:mean_subG}, with
probability at least $1-\delta$, we have
\begin{align}
    \bigg\{\sum_{\ell=0}^L \| \Proj_{\mathscr U_\ell}
    \barxi_{\calS^{(\ell) }_\calI}\big\|_2^2\bigg\}^{1/2}
    &\le
    \bigg\{\sum_{\ell=0}^L \bigg( \frac{\|\Proj_{\mathscr U_\ell}\bar\bxi_n
    \|_2}{1 - \eta}  +  \frac{{\sf r}_{n, \mathfrak{s}}\sqrt{\eta} + 2\eta\sqrt{2 \log(2e/\eta)}}{1-\eta}\bigg)^2\bigg\}^{1/2}\\
    &\le \bigg\{\sum_{\ell=0}^L \big( 4\|\Proj_{\mathscr U_\ell}
    \bar\bxi_n \|_2  +  4{\sf r}_{n, \mathfrak{s}}\sqrt{\eta} + 10\eta\sqrt{
    \log(2/\eta)}\big)^2\bigg\}^{1/2}\\
    &\le 4\|\barxi_n \|_2 + 4{\sf r}_{n, \mathfrak{s}}\sqrt{\eta L} +
    10\eta\sqrt{L\log(2/\eta)}.
\end{align}
Moreover, since the random variable $\barxi_n$ is sub-Gaussian with variance proxy $\mathfrak{s}/n$, hence by \Cref{lem:sigma_p} we have
\begin{align}
    \|\barxi_n\|_2^2\le\frac{16\mathfrak{s}\big(\sqrt{p}+2\sqrt{\log(2/\delta)}\big)^2}{n}
    ={\sf r}_{n, \mathfrak{s}}^2
\end{align}
with probability at least $1-\delta$. Therefore, with
probability at least $1-2\delta$,
\begin{align}
    \bigg(\sum_{\ell=0}^L \| \Proj_{\mathscr U_\ell}
    \barxi_{\calS^{(\ell) }_\calI}\big\|_2^2\bigg)^{1/2}
    &\le  4{\sf r}_{n, \mathfrak{s}} \big(1 + \sqrt{L\eta}\big) + 10\eta
    \sqrt{L\log(2/\eta)}.\label{eq:14bisc}
\end{align}

Next, the combination of \Cref{lem:geommed_subG} and \Cref{lem:filternoncontamine_subG} and the fact that
$p_L = \dim(\mathscr U_L) = 1$  imply that with probability at least $1-\delta$
\begin{align}
    \big\|\Proj_{\mathscr
    U_L}\bmu^* -\med_{\mathscr U_L}\big\|_2 \le
    \frac{2(1+ \sfC{\sf r}_{n, \mathfrak{s}} \sqrt{\mathfrak{s}}/4)}{1-2\varepsilon},\label{eq:15c}
\end{align}
where $\sfC$ is the same universal constant as in \Cref{lem:vershynin:bis}. Recall that we have chosen $t$ in such a way that
\begin{align}
    t \le \frac{3(1 + \sfC_0{\sf r}_{n, \mathfrak{s}} \sqrt{\mathfrak{s}}/4\sqrt{\tau})}{1-2\varepsilon^*}
        +  1.6\sfC_0{\mathfrak{s}}\sqrt{\log(2/{{\tau}})}.\label{eq:16c}
\end{align}
Combining \eqref{eq:14c}, \eqref{eq:14bisc}, \eqref{eq:15c} and \eqref{eq:16c},
we arrive at the inequality
\begin{align}
    \big\|\bmuSDR - \bmu^*\big\|_2
    &\le 2\sqrt{2\varepsilon L T_1} + 9\varepsilon t
    \sqrt{L} +  \frac{8\varepsilon(1+\sfC{\sf r}_{n, \mathfrak{s}} \sqrt{\mathfrak{s}}/4)}{1-2\varepsilon}
    + 4{\sf r}_{n, \mathfrak{s}} \big(1 + \sqrt{L\eta}\big) + 10\eta
    \sqrt{L\log(2/\eta)}\\
    &\le 2\sqrt{2\varepsilon L T_1} +
    \frac{27\varepsilon \sqrt{L}(1 + \sfC{\sf r}_{n, \mathfrak{s}} \sqrt{\mathfrak{s}}/4\sqrt{\tau})}{1 -
    2\varepsilon^*} + 14.4\sfC{\mathfrak{s}}\varepsilon \sqrt{L\log(2/\tau)} \\
    &\quad +  \frac{8\varepsilon(1+\sfC{\sf r}_{n, \mathfrak{s}} \sqrt{\mathfrak{s}}/4)}{1-2\varepsilon}
    + 4{\sf r}_{n, \mathfrak{s}} \big(1 + \sqrt{L\eta}\big) + 10\eta
    \sqrt{L\log(2/\eta)}
\end{align}
that holds with probability at least $1-3\delta$. In the
upper bound obtained above, only the term $T_1$ remains
random. We can upper bound this term using
\Cref{lem:sigma_subG}. It implies that with probability
at least $1-2\delta$, we have
\begin{align}
    T_1 &\le  A\frac{\sqrt{np} + p + n\eta
    \log(2e/\eta) + 2\log(1/\delta)}{n(1-\eta)} + \big(
    4{\sf r}_{n, \mathfrak{s}}(1+\sqrt{\eta}) + 10\eta\sqrt{\log(2/\eta)}
    \big)^2 + \varepsilon\\
    &\le A_{\mathfrak{s}}\big({\sf r}_{n, \mathfrak{s}} + {\sf r}_{n, \mathfrak{s}}^2  + 8\eta
    \log(2/\eta)\big) + \big (7.5{\sf r}_{n, \mathfrak{s}} + 10\eta
    \sqrt{\log(2/\eta)}\big)^2 + \varepsilon,
\end{align}
where $A_{\mathfrak{s}}$ is a constant depending only on the variance proxy $\mathfrak{s}$, the value of which is not necessarily the same in further simplifications of the expression from the last display. Then, using the triangle inequality several times we arrive at the following expression
\begin{align}
    \sqrt{\varepsilon T_1} &\le  \big\{A_{\mathfrak{s}}\varepsilon
    \big({\sf r}_{n, \mathfrak{s}} + {\sf r}_{n, \mathfrak{s}}^2  + 8\eta \log
    (2/\eta) \big)\big\}^{1/2} + \big (7.5{\sf r}_n +
    10\eta \sqrt{\log(2/\eta)}\big)\sqrt{\varepsilon} +
    \varepsilon\\
    & \le (A_{\mathfrak{s}} + \sqrt{A_{\mathfrak{s}}/2} + 5.4){\sf r}_{n, {\mathfrak{s}}}+ (7.1 + 2\sqrt{2A_{\mathfrak{s}}})
    \tau \sqrt{\log(2/\tau)} + (9.1 + 2\sqrt{2A_{\mathfrak{s}}})\varepsilon\sqrt{
    \log(2/\varepsilon)}.
\end{align}
These inequalities imply that there is a universal
constant $\sfC$ such that
\begin{align}\label{eq:18c}
    \big\|\bmuSDR - \bmu^*\big\|_2
    & \le \frac{\sfC \mathfrak{s} \big( A_{\mathfrak{s}}{\sf r}_{n, \mathfrak{s}}/\sqrt{\mathfrak{s}} + \tau\sqrt{\log(2/\tau)}
    + \varepsilon\sqrt{\log(2/\varepsilon)} + {\sf r}_{n, \mathfrak{s}}
    \varepsilon/\sqrt{\mathfrak{s}\tau} \big) \sqrt{L}}{1-2\varepsilon^*}.
\end{align}
Let us denote $\log_+(x) = \max\{0, \log(x)\}$ the positive part of logarithm, then we choose
\begin{align}
    \tau = \frac14\bigwedge \frac{\bar{\sf r}_{n,\mathfrak{s}}}{\sqrt{\log_+(2/\bar{\sf r}_{n,\mathfrak{s}})}}, \qquad\text{with}\qquad
    \bar{\sf r}_{n,\mathfrak{s}} = \frac{3\sqrt{\mathfrak{s}}\big(\sqrt{p} + 2\sqrt{\log(2/\delta)}\big)}{\sqrt{n}}.
\end{align}
Note that ${\sf r}_{n,\mathfrak{s}} \le \sqrt{2} \bar{\sf r}_{n,\mathfrak{s}}$ and,
furthermore, $\tau = 1/4$ whenever
$\bar{\sf r}_{n,\mathfrak{s}} \ge 1/2$. Therefore,
${\sf r}_{n,\mathfrak{s}}  \varepsilon/\sqrt{\tau}\le {\sf r}_{n,\mathfrak{s}} + \varepsilon$.
Inserting this value of $\tau$ in \eqref{eq:18c} leads to
\begin{align}
    \big\|\bmuSDR - \bmu^*\big\|_2
    &\le \frac{\sfC \, {\mathfrak{s}} \big( A_{\mathfrak{s}}{\sf r}_{
    n, \mathfrak{s}}/\sqrt{\mathfrak{s}} + \varepsilon\sqrt{\log(2/
    \varepsilon)}\big) \sqrt{L}}{1-2\varepsilon^*}.
\end{align}
where $\sfC$ is a universal constant. Replacing ${\sf r}_{n, \mathfrak{s}}$
by its expression, and upper bounding $L$ by $2\log p$,  we arrive at
\begin{align}\label{eq:17}
    \big\|\bmuSDR - \bmu^*\big\|_2
    &\le \frac{\sfC \, {\mathfrak{s}}\, \sqrt{\log p}}{1-2\varepsilon^*}
    \bigg( A_{\mathfrak{s}}\sqrt{\frac{p}{n}}+ \varepsilon\sqrt{
    \log(2/\varepsilon)} + A_{\mathfrak{s}}\sqrt{\frac{\log(1/\delta)}{n}}
    \bigg).
\end{align}
Note that this inequality holds true on an event of probability
at least $1-5\delta$.

\begin{funding}
The authors were supported by the grant Investissements
d’Avenir (ANR-11-IDEX-0003/Labex Ecodec/ANR-11-LABX-0047) and
by the FAST Advance grant. The third author was supported in
part by the center Hi!PARIS.
\end{funding}


\appendix

{
\renewcommand{\addtocontents}[2]{}
\bibliography{LiteratureRobust}
}

\end{document}